\documentclass[10pt]{amsart}
\usepackage{amssymb,amstext,amsmath,amscd,amsthm,amsfonts,enumerate,graphicx,latexsym,stmaryrd,ascmac,url}
\usepackage[usenames]{color}
\usepackage[all]{xy}
\newtheorem{thm}{Theorem}[section]
\newtheorem{lem}[thm]{Lemma}
\newtheorem{prop}[thm]{Proposition}
\newtheorem{cor}[thm]{Corollary}
\theoremstyle{definition}
\newtheorem{dfn}[thm]{Definition}
\newtheorem{ques}[thm]{Question}
\newtheorem{conv}[thm]{Convention}
\newtheorem{conj}[thm]{Conjecture}
\newtheorem{rem}[thm]{Remark}
\newtheorem{ex}[thm]{Example}
\newtheorem{proposal}[thm]{Proposal}
\newtheorem{claim}{Claim}
\newtheorem{setup}[thm]{Setup}

\renewcommand{\qedsymbol}{$\blacksquare$}
\numberwithin{equation}{thm}
\def\A{\mathrm{A}}
\def\a{\mathsf{A}}
\def\add{\operatorname{\mathsf{add}}}
\def\ann{\operatorname{ann}}
\def\C{\mathcal{C}}
\def\CC{\mathbb{C}}
\def\cm{\operatorname{\mathsf{CM}}}
\def\cmp{\operatorname{\mathsf{CM}_{\scalebox{0.6}{\mbox{\boldmath$+$}}}}}
\def\cmz{\operatorname{\mathsf{CM}_0}}
\def\codepth{\operatorname{codepth}}
\def\cok{\operatorname{Cok}}
\def\cx{\operatorname{cx}}
\def\D{\mathrm{D}}
\def\depth{\operatorname{depth}}
\def\ds{\operatorname{\mathsf{D_{sg}}}}
\def\E{\mathrm{E}}
\def\e{\operatorname{e}}
\def\edim{\operatorname{edim}}
\def\End{\operatorname{End}}
\def\Ext{\operatorname{Ext}}
\def\fitt{\operatorname{Fitt}}
\def\grade{\operatorname{grade}}
\def\height{\operatorname{ht}}
\def\Hom{\operatorname{\mathsf{Hom}}}
\def\I{\operatorname{I}}
\def\ind{\operatorname{\mathsf{ind}}}
\def\ker{\operatorname{Ker}}
\def\lcm{\operatorname{\mathsf{\underline{CM}}}}
\def\lend{\operatorname{\underline{End}}}
\def\lmod{\operatorname{\mathsf{\underline{mod}}}}
\def\m{\mathfrak{m}}
\def\Min{\operatorname{Min}}
\def\mod{\operatorname{\mathsf{mod}}}
\def\N{\mathcal{N}}
\def\n{\mathfrak{n}}
\def\nf{\operatorname{NF}}
\def\nzd{\operatorname{NZD}}
\def\p{\mathfrak{p}}
\def\pd{\operatorname{pd}}
\def\Q{\mathrm{Q}}
\def\q{\mathfrak{q}}
\def\R{\mathbb{R}}
\def\r{\mathfrak{r}}

\def\RR{\mathrm{R}}
\def\sing{\operatorname{Sing}}
\def\spec{\operatorname{Spec}}
\def\SS{\mathcal{S}}
\def\supp{\operatorname{Supp}}
\def\syz{\mathrm{\Omega}}
\def\Tor{\operatorname{\mathsf{Tor}}}
\def\v{\operatorname{V}}
\def\X{\mathcal{X}}
\def\XX{\mathsf{X}}
\def\Y{\mathcal{Y}}

\begin{document}
\allowdisplaybreaks
\title[MCM modules that are not locally free on the punctured spectrum]{Maximal Cohen--Macaulay modules that are not locally free on the punctured spectrum}
\author{Toshinori Kobayashi}
\address[T. Kobayashi]{Graduate School of Mathematics, Nagoya University, Furocho, Chikusaku, Nagoya, Aichi 464-8602, Japan}
\email{m16021z@math.nagoya-u.ac.jp}
\author{Justin Lyle}
\address[J. Lyle]{Department of Mathematics, University of Kansas, Lawrence, KS 66045-7523, USA}
\email{justin.lyle@ku.edu}
\urladdr{http://people.ku.edu/~j830l811/}
\author{Ryo Takahashi}
\address[R. Takahashi]{Graduate School of Mathematics, Nagoya University, Furocho, Chikusaku, Nagoya, Aichi 464-8602, Japan/Department of Mathematics, University of Kansas, Lawrence, KS 66045-7523, USA}
\email{takahashi@math.nagoya-u.ac.jp}
\urladdr{https://www.math.nagoya-u.ac.jp/~takahashi/}
\subjclass[2010]{13C60, 13H10, 16G60}
\keywords{Cohen--Macaulay ring, Gorenstein ring, hypersurface, isolated singularity, maximal Cohen--Macaulay module, punctured spectrum, representation type, singular locus}
\thanks{TK was partly supported by JSPS Grant-in-Aid for JSPS Fellows 18J20660 and JSPS Overseas Challenge Program for Young Researchers.
RT was partly supported by JSPS Grant-in-Aid for Scientific Research 16K05098 and JSPS Fund for the Promotion of Joint International Research 16KK0099}
\begin{abstract}
We say that a Cohen--Macaulay local ring has finite $\cmp$-representation type if there exist only finitely many isomorphism classes of indecomposable maximal Cohen--Macaulay modules that are not locally free on the punctured spectrum.
In this paper, we consider finite $\cmp$-representation type from various points of view, relating it with several conjectures on finite/countable Cohen--Macaulay representation type.
We prove in dimension one that the Gorenstein local rings of finite $\cmp$-representation type are exactly the local hypersurfaces of countable $\cm$-representation type, that is, the hypersurfaces of type $(\A_\infty)$ and $(\D_\infty)$.
We also discuss the closedness and dimension of the singular locus of a Cohen--Macaulay local ring of finite $\cmp$-representation type.
\end{abstract}
\maketitle
\section{Introduction}

Cohen--Macaulay representation theory has been studied widely and deeply for more than four decades.
The theorems of Herzog \cite{H} in the 1970s and of Buchweitz, Greuel and Schreyer \cite{BGS} in the 1980s are recognized as some of the most crucial results in this long history of Cohen--Macaulay representation theory.
Both are concerned with Cohen--Macaulay local rings of finite/countable $\cm$-representation type, that is, Cohen--Macaulay local rings possessing finitely/infinitely-but-countably many nonisomorphic indecomposable maximal Cohen--Macaulay modules.
Herzog proved that quotient singularities of dimension two have finite $\cm$-representation type and that Gorenstein local rings of finite $\cm$-representation type are hypersurfaces.
Buchweitz, Greuel and Schreyer proved that the local hypersurfaces of finite (resp. countable) $\cm$-representation type are precisely the local hypersurfaces of type $(\A_n)$ with $n\ge1$, $(\D_n)$ with $n\ge4$, and $(\E_n)$ with $n=6,7,8$ (resp. $(\A_\infty)$ and $(\D_\infty)$).

At the beginning of this century, Huneke and Leuschke \cite{HL02} proved that Cohen--Macaulay local rings of finite $\cm$-representation type have isolated singularities.
However, there are ample examples of Cohen--Macaulay local rings not having isolated singularities, including the local hypersurfaces of type $(\A_\infty)$ and $(\D_\infty)$ appearing above.
Cohen--Macaulay representation theory for non-isolated singularities has been studied by many authors so far; see \cite{AKM,BD,HN,IW} for instance.
It should be remarked that a Cohen--Macaulay local ring with a non-isolated singularity always admits maximal Cohen--Macaulay modules that are {\em not} locally free on the punctured spectrum.
Focusing on these modules, Araya, Iima and Takahashi \cite{hsccm} found out that the local hypersurfaces of type $(\A_\infty)$ and $(\D_\infty)$ have {\em finite $\cmp$-representation type}, that is, there exist only finitely many isomorphism classes of indecomposable maximal Cohen--Macaulay modules that are {\em not} locally free on the punctured spectrum.

In this paper, we investigate Cohen--Macaulay local rings of finite $\cmp$-representation type from various viewpoints.
Our basic landmark is the following conjecture, which includes the converse of the result of Araya, Iima and Takahashi stated above.
We shall give positive results to this conjecture.

\begin{conj} \label{conj11}
Let $R$ be a complete local Gorenstein ring of dimension $d$ not having an isolated singularity.
Then the following two conditions are equivalent.
\begin{enumerate}[\rm(1)]
\item
The ring $R$ has finite $\cmp$-representation type.
\item The ring $R$ has countable $\cm$-representation type.

\end{enumerate}
\end{conj}

Combining the result of Buchweitz, Greuel and Schreyer, this
conjecture says that, when $R$ is a hypersurface having an uncountable
algebraically closed coefficient field of characteristic not $2$,
condition (2) is equivalent to $R$ being an $(\A_{\infty})$ or
$(\D_{\infty})$ singularity. In this setting, the implication
$(2)\Rightarrow(1)$ holds by \cite[Proposition 2.1]{hsccm}.

From now on, we state our main results and the organization of this paper.
Section \ref{p} is devoted to a couple of preliminary definitions and lemmas, while Section \ref{c} presents some conjectures and questions on finite/countable $\cm$-representation type.
Our results are stated in the later sections.
In what follows, let $R$ be a Cohen--Macaulay local ring.

In Section \ref{s}, we consider the (Zariski-)closedness and (Krull) dimension of the singular locus $\sing R$ of $R$ in connection with the works of Huneke and Leuschke \cite{HL02,HL03}.
As we state above, they proved in \cite{HL02} that if $R$ has finite $\cm$-representation type, then it has an isolated singularity, i.e., $\sing R$ has dimension at most zero.
Also, they showed in \cite{HL03} that if $R$ is complete or has uncountable residue field, and has countable $\cm$-representation type, then $\sing R$ has dimension at most one.
In relation to these results, we prove the following theorem, whose second assertion extends the result of Huneke and Leuschke \cite{HL03} from countable $\cm$-representation type to countable $\cmp$-representation type (i.e., having infinitely but countably many nonisomorphic indecomposable maximal Cohen--Macaulay modules that are not locally free on the punctured spectrum).

\begin{thm}[Theorem \ref{12} and Corollary \ref{16}]
Let $(R,\m,k)$ be a Cohen--Macaulay local ring.
\begin{enumerate}[\rm(1)]
\item
Suppose that $R$ has finite $\cmp$-representation type.
Then the singular locus $\sing R$ is a finite set.
Equivalently, it is a closed subset of $\spec R$ with dimension at most one.
\item
Suppose that $R$ has countable $\cmp$-representation type.
Then the set $\sing R$ is at most countable.
It has dimension at most one if $R$ is either complete or $k$ is uncountable.
\end{enumerate}
\end{thm}

Furthermore, Huneke and Leuschke \cite{HL03} proved that if $R$ admits a canonical module and has countable $\cm$-representation type, then the localization $R_\p$ at each prime ideal $\p$ of $R$ has at most countable $\cm$-representation type as well.
We prove a result on finite $\cmp$-representation type in the same context.

\begin{thm}[Theorem \ref{45}]
Let $(R,\m)$ be a Cohen--Macaulay local ring with a canonical module.
Suppose that $R$ has finite $\cmp$-representation type.
Then $R_\p$ has finite $\cm$-representation type for all $\p\in \spec R\setminus\{\m\}$.
In particular, $R_\p$ has finite $\cmp$-representation type for all $\p\in\spec R$.
\end{thm}

In Section \ref{n} we provide various necessary conditions for a given Cohen--Macaulay local ring to have finite $\cmp$-representation type.

\begin{thm}[Theorem \ref{19}]
Let $(R,\m)$ be a Cohen--Macaulay local ring of dimension $d>0$.
Let $I$ be an ideal of $R$ such that $R/I$ is maximal Cohen--Macaulay over $R$.
Then $R$ has infinite $\cmp$-representation type in each of the following cases.
\begin{enumerate}[\rm(1)]
\item
The ring $R/I$ has infinite $\cmp$-representation type.
\item
The set $\v(I)$ is contained in $\v(0:I)$, and either $R/I$ has infinite $\cm$-representation type or $d\ge2$.
\item
The ideal $I+(0:I)$ is not $\m$-primary, $R/I$ has infinite $\cm$-representation type, and $R/I$ is either Gorenstein, a domain, or analytically unramified with $d=1$.
\end{enumerate}
\end{thm}

This theorem may look technical, but it actually gives rise to a lot of restrictions which having finite $\cmp$-representation type produces, and is used in the later sections. One concrete example where Theorem \ref{19} applies is when $I=(x)$ and $(0:x)=(x)$; see Corollary \ref{30}.
Here we introduce one of the applications of the above theorem.
Denote by $\cm(R)$ the category of maximal Cohen--Macaulay $R$-modules, and by $\ds(R)$ the singularity category of $R$.

\begin{thm}[Theorem \ref{15}]
Let $R$ be a Cohen--Macaulay local ring of dimension $d>0$.
Let $I$ be an ideal of $R$ with $\v(I)\subseteq\v(0:I)$ such that $R/I$ is maximal Cohen--Macaulay over $R$.
Suppose that $R$ has finite $\cmp$-representation type.
Then one must have $d=1$.
If $I^n=0$ for some integer $n>0$, then $\cm(R)$ has dimension at most $n-1$ in the sense of \cite{dim}.
If $R$ is Gorenstein, then $R$ is a hypersurface and $\ds(R)$ has dimension at most $n-1$ in the sense of \cite{R}.
\end{thm}

There are folklore conjectures that a Gorenstein local ring of countable $\cm$-representation type is a hypersurface, and that, for a Cohen--Macaulay local ring $R$ of countable $\cm$-representation type, $\cm(R)$ has dimension at most one.
The above theorem gives partial answers to the variants of these folklore conjectures for finite $\cmp$-representation type.

In Section \ref{m}, we prove the following, which characterizes the Gorenstein rings or finite $\cmp$-representation type not having an isolated singularity in the dimension $1$ case. This theorem has the consequence of answering Conjecture \ref{conj11} in the affirmative when $R$ has an uncountable algebraically closed coefficient field of characteristic not equal to $2$.

\begin{thm}[Theorem \ref{46}]
Let $R$ be a homomorphic image of a regular local ring.
Suppose that $R$ does not have an isolated singularity but is Gorenstein.
If $\dim R=1$, the following are equivalent.
\begin{enumerate}[\rm(1)]
\item
The ring $R$ has finite $\cmp$-representation type.
\item
There exist a regular local ring $S$ and a regular system of parameters $x,y$ such that $R$ is isomorphic to $S/(x^2)$ or $S/(x^2y)$.
\end{enumerate}
When either of these two conditions holds, the ring $R$ has countable $\cm$-representation type.
\end{thm}

In Section \ref{h}, we explore the higher-dimensional case, that is, we try to understand the Cohen--Macaulay local rings $R$ of finite $\cmp$-representation type in the case where $\dim R\ge2$.
We prove the following two results in this section.

\begin{thm}[Corollary \ref{69}]
Let $R$ be a complete local hypersurface of dimension $d\ge2$ which is not an integral domain.
Suppose that $R$ has finite $\cmp$-representation type.
Then one has $d=2$, and there exist a regular local ring $S$ and elements $x,y\in S$ with $R\cong S/(xy)$ such that $S/(x)$ and $S/(y)$ have finite $\cm$-representation type and $S/(x,y)$ is an integral domain of dimension $1$.
\end{thm}

\begin{thm}[Corollaries \ref{55} and \ref{56}]
Let $R$ be a $2$-dimensional non-normal Cohen--Macaulay complete local domain.
Suppose that $R$ has finite $\cmp$-representation type.
Then the integral closure $\overline R$ of $R$ has finite $\cm$-representation type.
If $R$ is Gorenstein, then $R$ is a hypersurface.
\end{thm}

The former theorem gives a strong restriction of the structure of a hypersurface of finite $\cmp$-representation type which is not an integral domain.
The latter theorem supports the conjecture that a Gorenstein local ring of finite $\cmp$-representation type is a hypersurface.
Note that, under the assumption of the theorem plus the assumption that $R$ is equicharacteristic zero, the integral closure $\overline R$ is a quotient surface singularity by the theorem of Auslander \cite{Aus} and Esnault \cite{Es}.

\section{Preliminaries}\label{p}

This section is devoted to stating our conventions, and to recalling the definitions of the notions which repeatedly appear in this paper.

\begin{conv}
Throughout this paper, unless otherwise specified, we adopt the following convention.
Rings are commutative and noetherian, and modules are finitely generated.
Subcategories are full and strict (i.e., closed under isomorphism).
Subscripts and superscripts are often omitted unless there is a risk of confusion.
An identity matrix of suitable size is denoted by $E$.
\end{conv}

\begin{dfn}
Let $R$ be a ring.
\begin{enumerate}[(1)]
\item
An $R$-module $M$ is {\em maximal Cohen--Macaulay} if the inequality $\depth M_\p\ge\dim R_\p$ holds for all $\p\in\spec R$.
Hence, by definition, the zero module is maximal Cohen--Macaulay.
\item
We denote by $\mod R$ the category of (finitely generated) $R$-modules, and by $\cm(R)$ the subcategory of $\mod R$ consisting of maximal Cohen--Macaulay $R$-modules.
For a subcategory $\X$ of $\mod R$, we denote by $\ind\X$ the set of isomorphism classes of indecomposable $R$-modules in $\X$, and by $\add_R\X$ the {\em additive closure} of $\X$, that is, the subcategory of $\mod R$ consisting of direct summands of finite direct sums of objects in $\X$.
\item
A subset $S$ of $\spec R$ is called {\em specialization-closed} if $\v(\p)\subseteq S$ for all $\p\in S$.
This is equivalent to saying that $S$ is a union of closed subsets of $\spec R$ in the Zariski topology.
\item
Let $S$ be a subset of $\spec R$.
Then it is easy to see that
$$
\sup\{\dim R/\p\mid\p\in S\}\ge\sup\{n\ge0\mid\text{there exists a chain $\p_0\subsetneq\p_1\subsetneq\cdots\subsetneq\p_n$ in $S$}\},
$$
and the equality holds if $S$ is specialization-closed.
The {\em (Krull) dimension} of a specialization-closed subset $S$ of $\spec R$ is defined as this common number and denoted by $\dim S$.
\item
The {\em singular locus} of $R$, denoted by $\sing R$, is by definition the set of prime ideals $\p$ of $R$ such that $R_\p$ is not a regular local ring.
It is clear that $\sing R$ is a specialization-closed subset of $\spec R$.
If $R$ is excellent, then by definition $\sing R$ is a closed subset of $\spec R$ in the Zariski topology.
\item
For an $m\times n$ matrix $A$ over $R$, we denote by $\cok_RA$ the cokernel of the map $R^{\oplus n}\to R^{\oplus m}$ given by $x\mapsto Ax$, and by $\I_s(A)$ the ideal of $R$ generated by all the $s$-minors of $A$.
\item
For an $R$-module $M$, we denote by $\fitt_r(M)$ is the $r$th {\em Fitting invariant} of $M$, that is, we have $\fitt_r(M)=\I_{m-r}(A)$ if there exists an exact sequence $R^{\oplus n}\xrightarrow{A}R^{\oplus m}\to M\to0$.
\end{enumerate}
\end{dfn}

\begin{dfn}
Let $(R,\m,k)$ be a local ring.
\begin{enumerate}[(1)]
\item
For an $R$-module $M$, we denote by $\nu_R(M)$ the minimal number of generators of $M$, that is, $\nu_R(M)=\dim_k(M\otimes_Rk)$.
\item
Let $M$ an $R$-module and $n\ge0$ an integer.
We denote by $\syz_R^nM$ (or simply $\syz^nM$) the $n$-th {\em syzygy} of $M$, i.e., the image of the $n$-th differential map in the minimal free resolution of $M$.
This is uniquely determined up to isomorphism.
\item
We denote by $\edim R$ the embedding dimension of $R$, and by $\codepth R$ the codepth of $R$, i.e., $\codepth R=\edim R-\depth R$.
We say that $R$ is a {\em hypersurface} if $\codepth R\le1$.
\item
An $R$-module $M$ is called {\em periodic} if $\syz^eM\cong M$ for some $e>0$.
\item
The {\em complexity} of an $R$-module $M$, denoted by $\cx_RM$, is defined as the infimum of nonnegative integers $n$ such that there exists a real number $r$ satisfying the inequality $\beta_i^R(M)\le ri^{n-1}$ for $i\gg 0$, where $\beta_i^R(M)$ stands for the $i$th Betti number of $M$.
\item
The {\em Loewy length} of $R$ is defined by $\ell\ell(R)=\inf\{n\ge0\mid\m^n=0\}$.
Note that $\ell\ell(R)<\infty$ if and only if $R$ is artinian.
\end{enumerate}
\end{dfn}

\begin{dfn}
Let $R$ be a local ring.
\begin{enumerate}[(1)]
\item
For a subcategory $\X$ of $\mod R$ we denote by $[\X]$ the smallest subcategory of $\mod R$ containing $R$ and $\X$ that is closed under finite direct sums, direct summands and syzygies, i.e., $[\X]=\add_R(\{R\}\cup\{\syz^iX\mid i\ge0,\,X\in\X\})$.
When $\X$ consists of a single object $X$, we simply denote it by $[X]$.
\item
For subcategories $\X,\Y$ of $\mod R$ we denote by $\X\circ\Y$ the subcategory of $\mod R$ consisting of objects $M$ which fit into an exact sequence $0 \to X \to M \to Y \to 0$ in $\mod R$ with $X\in\X$ and $Y\in\Y$.
We set $\X\bullet\Y=[[\X]\circ[\Y]]$.
\item
Let $\C$ be a subcategory of $\mod R$.
Put
$$
[\C]_r=
\begin{cases}
\{0\} & (r=0),\\
[\C] & (r=1),\\
{[\C]}_{r-1}\bullet\C=[{[\C]}_{r-1}\circ[\C]] & (r\ge2).
\end{cases}
$$
If $\C$ consists of a single object $C$, then we simply denote it by ${[C]}_r$.
\item
Let $\X$ be a subcategory of $\mod R$.
We define the {\it dimension} of $\X$, denoted by $\dim\X$, as the infimum of the integers $n\ge0$ such that $\X={[G]}_{n+1}$ for some $G\in\X$.
\end{enumerate}
\end{dfn}

\section{Conjectures and questions}\label{c}

In this section, we present several conjectures and questions which we deal with in later sections.
First of all, let us give several definitions of representation types, including that of finite $\cmp$-representation type, which is the main subject of this paper.

\begin{dfn}
Let $R$ be a Cohen--Macaulay ring.
By $\cmz(R)$ we denote the subcategory of $\cm(R)$ consisting of modules that are locally free on the punctured spectrum of $R$, and set
$$
\cmp(R):=\cm(R)\setminus\cmz(R).\footnote{The index $0$ (resp. $+$) in $\cmz(R)$ (resp. $\cmp(R)$) means that it consists of modules whose nonfree loci have zero (resp. positive) dimension.}
$$
For each $\XX\in\{\cm,\cmz,\cmp\}$ we say that $R$ has {\em finite} (resp. {\em countable}) $\XX$-representation type if there exist only finitely (resp. countably) many isomorphism classes of indecomposable modules in $\XX(R)$.
We say that $R$ has {\em infinite} (resp. {\em uncountable}) $\XX$-representation type if $R$ does not have finite (resp. countable) $\XX$-representation type.
Also, $R$ is said to have {\em bounded} $\XX$-representation type if there exists an upper bound of the multiplicities of indecomposable modules in $\XX(R)$, and said to have {\em unbounded} $\XX$-representation type if $R$ does not have bounded $\XX$-representation type.
\end{dfn}

Let $R$ be a complete local hypersurface with uncountable algebraically closed coefficient field of characteristic not two.
Buchweitz, Greuel and Schreyer \cite[Theorem B]{BGS} (see also \cite[Theorem 14.16]{LW}) prove that $R$ has countable $\cm$-representation type if and only if it is either an $(\A_\infty)$-singularity or a $(\D_\infty)$-singularity.
Moreover, when this is the case, they give a complete classification of the indecomposable maximal Cohen--Macaulay $R$-modules.
Using this result, Araya, Iima and Takahashi \cite[Theorem 1.1 and Corollary 1.3]{hsccm} prove the following theorem (see \cite[Proposition 3.5(3)]{dim}), which provides examples of a Cohen--Macaulay local ring of finite $\cmp$-representation type.

\begin{thm}[Araya--Iima--Takahashi]\label{5}
Let $R$ be a complete local hypersurface with uncountable algebraically closed coefficient field of characteristic not two.
If $R$ has countable $\cm$-representation type, then the following statements hold.
\begin{enumerate}[\rm(1)]
\item
The ring $R$ has finite $\cmp$-representation type.
\item
There is an inequality $\dim\cm(R)\le1$.
\end{enumerate}
\end{thm}

By definition, there is a strong connection between finite $\cmp$-representation type and finite $\cm$-representation type.
The first assertion of Theorem \ref{5} suggests to us that finite $\cmp$-representation type should also be closely related to countable $\cm$-representation type.
Several conjectures have been presented so far concerning finite/countable $\cm$-representation type, and we set the following proposal.

\begin{proposal}\label{60}
One should consider the conjectures on finite/countable $\cm$-representation type for finite $\cmp$-representation type.
\end{proposal}

There has been a folklore conjecture on countable $\cm$-representation type probably since the 1980s.
Recently, this conjecture has been studied by Stone \cite{St}.

\begin{conj}\label{1}
A Gorenstein local ring $R$ of countable $\cm$-representation type is a hypersurface.
\end{conj}

This conjecture holds true if $R$ has finite $\cm$-representation type; see \cite[Theorem (8.15)]{Y}.
Also, the conjecture holds if $R$ is a complete intersection with algebraically closed uncountable residue field; see \cite[Existence Theorem 7.8]{AI}.
The following example shows that the assumption in the conjecture that $R$ is Gorenstein is necessary.

\begin{ex}
Let $S=\CC[\![x,y,z]\!]/(xy)$.
Then $S$ is an $(\a_\infty)$-singularity of dimension $2$, and has countable $\cm$-representation type by \cite[Theorem B]{BGS}.
Let $R$ be the second Veronese subring of $S$, that is, $R=\CC[\![x^2,xy,xz,y^2,yz,z^2]\!]\subseteq S$.
Then $R$ is a Cohen--Macaulay non-Gorenstein local ring of dimension $2$.
We claim that $R$ has countable $\cm$-representation type.
Indeed, let $N_1,N_2,\dots$ be the non-isomorphic indecomposable maximal Cohen--Macaulay $S$-modules.
Let $M$ be an indecomposable maximal Cohen--Macaulay $R$-module.
Then $N=\Hom_R(S,M)$ is a maximal Cohen--Macaulay $S$-module, and one can write $N\cong N_{a_1}^{\oplus b_1}\oplus\cdots\oplus N_{a_t}^{\oplus b_t}$.
Since $R$ is a direct summand of $S$, the module $M$ is a direct summand of $N$, and hence it is a direct summand of $N_{a_i}$ for some $i$.
The claim follows from this.
\end{ex}

Combining Conjecture \ref{1} with Proposal \ref{60} gives rise to the following question.

\begin{ques}\label{2}
Let $R$ be a Gorenstein local ring which is not an isolated singularity.
Suppose that $R$ has finite $\cmp$-representation type.
Then is $R$ a hypersurface?
\end{ques}

Here, the assumption that $R$ is not an isolated singularity is necessary.
Indeed, if $R$ is an isolated singularity, then $\#\ind\cmp(R)=0<\infty$.
We shall give answers to Question \ref{2} in Sections \ref{n} and \ref{h}.

Theorem \ref{5}(2) leads us to the following conjecture.

\begin{conj}\label{4}
Let $R$ be a Cohen--Macaulay local ring $R$ of countable $\cm$-representation type.
Then there is an inequality $\dim\cm(R)\le1$.
\end{conj}

This conjecture holds true if $R$ has finite $\cm$-representation type; see \cite[Proposition 3.7(1)]{dim}.
Let $R$ be a Gorenstein local ring.
Then the stable category $\lcm(R)$ of $\cm(R)$ is a triangulated category, and one can consider the {\em (Rouquier) dimension} $\dim\lcm(R)$ of $\lcm(R)$; we refer the reader to \cite{R} for the details.
One has $\dim\lcm(R)\le\dim\cm(R)$ with equality if $R$ is a hypersurface; see \cite[Proposition 3.5]{dim}.
There seems to be a folklore conjecture asserting that every (noncommutative) selfinjective algebra $\Lambda$ of tame representation type satisfies the inequality $\dim(\lmod\Lambda)\le1$.
So Conjecture \ref{4} is thought of as a Cohen--Macaulay version of this folklore conjecture.
Combining Conjecture \ref{4} with Proposal \ref{60} leads us to the following question.

\begin{ques}\label{61}
Let $R$ be a Cohen--Macaulay local ring of finite $\cmp$-representation type.
Then does one have $\dim\cm(R)\le1$?
\end{ques}

We shall give partial answers to this question in Section \ref{n}.

Huneke and Leuschke (\cite[Theorem 1.3]{HL03}) prove the following theorem, which solves a conjecture of Schreyer \cite[Conjecture 7.2.3]{Sc} presented in the 1980s.

\begin{thm}[Huneke--Leuschke]\label{7}
Let $(R,\m,k)$ be an excellent Cohen--Macaulay local ring.
Assume that $R$ is complete or $k$ is uncountable.
If $R$ has countable $\cm$-representation type, then $\dim\sing R\le1$.
\end{thm}

Indeed, the assumption that $R$ is excellent is unnecessary; see \cite[Theorem 2.4]{count}.
This result naturally makes us have the following question.

\begin{ques}\label{62}
Let $R$ be a Cohen--Macaulay local ring.
Suppose that $R$ has finite $\cmp$-representation type.
Then does $\sing R$ have dimension at most one?
\end{ques}

We shall give a complete answer to this question in the next Section \ref{s}.
In fact, we can even prove a stronger statement.

\section{The closedness and dimension of the singular locus}\label{s}

In this section, we discuss the structure of the singular locus of a Cohen--Macaulay local ring of finite $\cmp$-representation type.
First, we consider what the finiteness of the singular locus means.

\begin{lem}\label{14}
Let $R$ be a local ring with maximal ideal $\m$.
The following are equivalent.
\begin{enumerate}[\rm(1)]
\item
$\sing R$ is a finite set.
\item
$\sing R$ is a closed subset of $\spec R$ in the Zariski topology, and has dimension at most one.
\end{enumerate}
\end{lem}

\begin{proof}
(2)$\implies$(1):
We find an ideal $I$ of $R$ such that $\sing R=\v(I)$.
As $\sing R$ has dimension at most one, so does the local ring $R/I$.
Hence $\spec R/I=\Min R/I\cup\{\m/I\}$, and this is a finite set.

(1)$\implies$(2):
Write $\sing R=\{\p_1,\dots,\p_n\}$.
As $\sing R$ is specialization-closed, it coincides with the finite union $\v(\p_1)\cup\cdots\cup\v(\p_n)$ of closed subsets of $\spec R$.
Hence $\sing R$ is closed.

To show the other assertion, we claim (or recall) that a local ring $R$ of dimension at least two possesses infinitely many prime ideals of height one.
Indeed, for any $x\in\m$ we have $\height(x)\le1$ by Krull's principal ideal theorem, that is, $(x)$ is contained in some prime ideal $\p$ with $\height\p\le1$.
This argument shows that $\m=\bigcup_{\p\in\spec R,\,\height\p\le1}\p$.
Now suppose that there exist only finitely many prime ideals of $R$ having height one.
Then, since the number of the minimal primes is finite, so is the number of prime ideals of height at most one.
Therefore the above union is finite, and by prime avoidance $\m$ is contained in some $\p\in\spec R$ with $\height\p\le1$.
This implies $\dim R\le1$, which is a contradiction.
Thus the claim follows.

Now, assume that $\sing R$ has dimension at least $2$.
Then $\dim R/\p\ge2$ for some $\p\in\sing R$.
The above claim shows that the ring $R/\p$ has infinitely many prime ideals of height one, which have the form $\q/\p$ with $\q\in\v(\p)$.
Then $\q$ is also in $\sing R$, and hence $\sing R$ contains infinitely many prime ideals.
This contradiction shows that the dimension of $\sing R$ is at most $1$.
\end{proof}

The following theorem clarifies a close relationship between finite/countable $\cmp$-representation type and finiteness/countablity of the singular locus.

\begin{thm}\label{12}
If $R$ is a Cohen--Macaulay local ring of finite (resp. countable) $\cmp$-representation type, then $\sing R$ is a finite (resp. countable) set.
\end{thm}

\begin{proof}
First, let us consider the case where $R$ has finite $\cmp$-representation type.
Write $\ind\cmp(R)=\{G_1,\dots,G_t\}$, and pick $\p\in\sing R\setminus\{\m\}$.
Set $C=\syz_R^d(R/\p)$.
We claim that $\p=\ann_R\Tor_1^R(C,C)$.
Indeed, $\Tor_1^R(C,C)$ is isomorphic to $T:=\Tor_{1+2d}^R(R/\p,R/\p)$, which is killed by $\p$.
Hence $\p$ is contained in the annihilator.
Also, $T_\p$ is isomorphic to $\Tor_{1+2d}^{R_\p}(\kappa(\p),\kappa(\p))$, which does not vanish as $\p$ belongs to the singular locus.
Hence $\p$ is in the support of $T$, and contains the annihilator.
Now the claim follows.

Note that $C_\p$ is stably isomorphic to $\syz^d_{R_\p}(\kappa(\p))$, which is not $R_\p$-free since $R_\p$ is singular.
This means that $C$ belongs to $\cmp(R)$, and we get an isomorphism $C\cong G_{l_1}^{\oplus a_1}\oplus\cdots\oplus G_{l_s}^{\oplus a_s}\oplus H$ with $s\ge1$ and $1\le l_1<\cdots<l_s\le t$ and $a_1,\dots,a_s\ge1$ and $H\in\cmz(R)$.
It is easy to see that
$$
\textstyle\p=\left(\bigcap_{1\le i,j\le s}\ann_R\Tor_1^R(G_{l_i},G_{l_j})\right)\cap\ann_R\Tor_1^R(H,M)
$$
for some $R$-module $M$.
Since a prime ideal is irreducible in general, $\p$ coincides with one of the annihilators in the right-hand side.
The module $H$ is locally free on the punctured spectrum, and $\ann_R\Tor_1^R(H,M)$ contains a power of $\m$.
As $\p$ is a nonmaximal prime ideal, it cannot coincide with $\ann_R\Tor_1^R(H,M)$.
We thus have $\p=\ann_R\Tor_1^R(G_{l_p},G_{l_q})$ for some $p,q$.
This shows that we have only finitely many such prime ideals $\p$.
Consequently, $\sing R\setminus\{\m\}$ is a finite set, and so is $\sing R$.

We can analogously deal with the case where $R$ has countable $\cmp$-representation type.
In this case, we can write $\ind\cmp(R)=\{G_1,G_2,G_3,\dots\}$, and for each $\p\in\sing R\setminus\{\m\}$ there exist $p,q$ such that $\p=\ann_R\Tor_1^R(G_{l_p},G_{l_q})$.
\end{proof}

Theorem \ref{12} yields the following corollary, which gives a complete answer to Question \ref{62}.
We should remark that the second assertion of the corollary highly refines Theorem \ref{7} due to Huneke and Leuschke.

\begin{cor}\label{16}
Let $R$ be a Cohen--Macaulay local ring.
\begin{enumerate}[\rm(1)]
\item
If $R$ has finite $\cmp$-representation type, then $\sing R$ is closed and has dimension at most one.
\item
Suppose that $R$ has countable $\cmp$-representation type.
\begin{enumerate}[\rm(a)]
\item
If $k$ is uncountable, then $\sing R$ has dimension at most one.
\item
If $R$ is complete, then $\sing R$ is closed and has dimension at most one.
\end{enumerate}
\end{enumerate}
\end{cor}

\begin{proof}
(1) The assertion follows from Theorem \ref{12} and Lemma \ref{14}.

(2) Theorem \ref{12} implies that $\sing R$ is a countable set.
Note that $\sing R$ is specialization-closed.
If $R$ is complete or $k$ is uncountable, then we can apply \cite[Lemma 2.2]{count} to deduce that $\dim R/\p\le1$ for all $\p\in\sing R$.
In case $R$ is complete, $\sing R$ is closed as well since $R$ is excellent.
\end{proof}

Next we investigate the relationship of finite $\cmp$-representation type with localization of the base ring at a prime ideal.  In particular, we prove the following theorem, which says that finite $\cmp$-representation type implies finite $\cm$-representation type on the punctured spectrum.
This especially shows that finite $\cmp$-representation type localizes, which should be compared with the result of Huneke and Leuschke \cite[Theorem 2.1]{HL03} asserting that countable $\cm$-representation type localizes under the same assumption as in this theorem.
This is also connected with the conjecture that a Cohen--Macaulay local ring with an isolated singularity having countable $\cm$-representation type has finite $\cm$-representation type \cite[Page 3006]{HL03}.

\begin{thm}\label{45}
Let $(R,\m)$ be a Cohen--Macaulay local ring with a canonical module $\omega$.
Suppose that $R$ has finite $\cmp$-representation type.
Then $R_\p$ has finite $\cm$-representation type for all $\p\in\spec R\setminus\{\m\}$.
\end{thm}

\begin{proof}
Assume that there exists a prime ideal $\p\ne\m$ such that $R_\p$ has infinite $\cm$-representation type.
Then the set $\ind\cm(R_\p)\setminus\{\omega_\p\}$ is infinite, and we can take an infinite subset $\N=\{N_1,N_2,N_3,\dots\}$.

Fix a module $N\in\N$.
Then we can choose an $R$-module $L$ such that $N\cong L_\p$.
Take a {\em maximal Cohen--Macaulay approximation} of $L$ over $R$, that is, a short exact sequence
$$
\sigma:0 \to Y \to X \to L \to 0
$$
of $R$-modules such that $X$ is maximal Cohen--Macaulay and $Y$ has finite injective dimension; see \cite[Theorem 1.1]{AB}.
Localization gives an exact sequence $\sigma_\p:0\to Y_\p\to X_\p\to N\to0$.
As $N$ is maximal Cohen--Macaulay, $Y_\p$ is a maximal Cohen--Macaulay $R_\p$-module of finite injective dimension.
It follows from \cite[Exercise 3.3.28(a)]{BH} that $Y_\p\cong\omega_\p^{\oplus n}$ for some $n\ge0$.
The exact sequence $\sigma_\p$ splits, and we get an isomorphism $X_\p\cong N\oplus\omega_\p^{\oplus n}$.
Note that $\omega_\p$ is an indecomposable $R_\p$-module.

Let $X=X_1\oplus\cdots\oplus X_m$ be a decomposition of $X$ into indecomposable $R$-modules.
Then there is an isomorphism $(X_1)_\p\oplus\cdots\oplus(X_m)_\p\cong N\oplus\omega_\p^{\oplus n}$.
For each $i$ write $(X_i)_\p=Z_i\oplus\omega_\p^{\oplus l_i}$ with $l_i\ge0$ an integer and $Z_i$ not containing $\omega_\p$ as a direct summand; then $Z_i$ is a maximal Cohen--Macaulay $R_\p$-module.
We get an isomorphism
$$
Z_1\oplus\cdots\oplus Z_m\oplus\omega_\p^{\oplus(l_1+\cdots+l_m)}\cong N\oplus\omega_\p^{\oplus n}.
$$
Since $\End_{R_{\p}}(\omega_{\p}) \cong R_{\p}$ is a local ring, the module $Z_1\oplus\cdots\oplus Z_m$ does not contain $\omega_\p$ as a direct summand by \cite[Lemma 1.2]{LW}, while $N$ is an indecomposable $R_\p$-module with $N\ncong\omega_\p$.
Further, \cite[Lemma 2.1]{LW} also implies that $Z_1\oplus\cdots\oplus Z_m\cong N$ and $l_1+\cdots+l_m=n$, so we may assume that $Z_1\cong N$ and $Z_2=\cdots=Z_m=0$.
We thus have that $(X_1)_\p\cong N\oplus\omega_\p^{\oplus l_1}$.

Suppose that $(X_1)_\p$ is $R_\p$-free.
Then so are $N$ and $\omega_\p$, and we have $N\cong R_\p\cong\omega_\p$, which contradicts the choice of $N$.
Hence $(X_1)_\p$ is not $R_\p$-free, which implies that $X_1\in\cmp(R)$.

Thus we have shown that for each integer $i\ge1$ there exist an integer $n_i\ge0$ and a module $C_i\in\ind\cmp(R)$ such that $(C_i)_\p\cong N_i\oplus\omega_\p^{\oplus n_i}$.
Assume that $C_i\cong C_j$ for some $i\ne j$.
Then $N_i\oplus\omega_\p^{\oplus n_i}\cong N_j\oplus\omega_\p^{\oplus n_j}$, and, appealing again to \cite[Lemma 1.2]{LW}, we see that $N_i\cong N_j$ (and $n_i=n_j$), contrary to the choice of $\N$.
Hence $C_i\ncong C_j$ for all $i\ne j$, and we conclude that $R$ has infinite $\cmp$-representation type.
This contradiction completes the proof of the theorem.
\end{proof}

\begin{rem}
In Corollary \ref{16}(1) we proved that the singular locus of a Cohen--Macaulay local ring of finite $\cmp$-representation type has dimension at most one.
As an application of Theorem \ref{45}, we can get another proof of this statement under the assumption that $R$ admits a canonical module.

Let $R$ be a $d$-dimensional Cohen--Macaulay local ring with a canonical module, and suppose that $R$ has finite $\cmp$-representation type.
Then $R_\p$ has finite $\cm$-representation type for all nonmaximal prime ideals $\p$ of $R$ by Theorem \ref{45}.
In particular, $R_\p$ has an isolated singularity for all such $\p$ by \cite[Corollary 2]{HL02}.
This implies that $R_\q$ is a regular local ring in codimension $d-2$, and therefore $\dim\sing R\le 1$.
\end{rem}

\section{Necessary conditions for finite CM$_+$-representation type}\label{n}

In this section, we explore necessary conditions for a Cohen--Macaulay local ring to have finite $\cmp$-representation type.
For this purpose we begin with stating and showing a couple of lemmas.

\begin{lem}\label{9}
Let $R$ be a local ring.
\begin{enumerate}[\rm(1)]
\item
The subcategory of $\mod R$ consisting of periodic modules is closed under finite direct sums: if the $R$-modules $M_1,\dots,M_n$ are periodic, then so is $M_1\oplus\cdots\oplus M_n$.
\item
Let $0\to M_1\to \cdots\to M_n \to 0$ be an exact sequence in $\mod R$.
Let $r\ge0$ and $1\le t\le n$ be integers.
If $\cx_R(M_i)\le r$ for all $1\le i\le n$ with $i\ne t$, then $\cx_R(M_t)\le r$.
\end{enumerate}
\end{lem}

\begin{proof}

(1) is obvious so we need only show (2), and it suffices to show the statement when $n=3$.
Suppose that $M_2,M_3$ have complexity at most $r$.
Then we find $p,q\in\R_{>0}$ such that $\beta_j^R(M_2)\le pj^{r-1}$ and $\beta_j^R(M_3)\le qj^{r-1}$ for $j\gg0$.
The induced exact sequence $\Tor_{j+1}^R(M_3,k) \to \Tor_j^R(M_1,k) \to \Tor_j^R(M_2,k)$ shows that $\beta_j^R(M_1)\le\beta_j^R(M_2)+\beta_{j+1}^R(M_3)\le(p+qr)j^{r-1}$ for $j\gg0$.
Therefore we obtain $\cx_R(M_3)\le r$.
The other cases are handled similarly.
\end{proof}

The subcategory $\cmp(R)$ of $\mod R$ is stable under syzygies.

\begin{lem}\label{8}
Let $R$ be a local ring.
Let $0\to N\to F\to M\to0$ be an exact sequence in $\mod R$ such that $F$ is free and $M$ is maximal Cohen--Macaulay.
Then $M$ belongs to $\cmp(R)$ if and only if so does $N$.
\end{lem}

\begin{proof}
Note that all the modules $N,F,M$ are maximal Cohen--Macaulay.
Hence the assertion is equivalent to saying that $M$ belongs to $\cmz(R)$ if and only if so does $N$.
The ``if'' part follows from the fact that $\cmz(R)$ is stable under syzygies.
To show the ``only if'' part, assume that $N$ is in $\cmz(R)$.
Let $\p$ be a nonmaximal prime ideal of $R$.
Then $N_\p$ is $R_\p$-free, and we see that the $R_\p$-module $M_\p$ has projective dimension at most $1$.
Note that $M_\p$ is maximal Cohen--Macaulay over $R_\p$.
The Auslander--Buchsbaum formula implies that $M_\p$ is free.
Hence $M$ is in $\cmz(R)$.
\end{proof}

We state some containments among indecomposable maximal Cohen--Macaulay modules over Cohen--Macaulay local rings, one of which is a homomorphic image of the other.

\begin{prop}\label{10}
Let $R$ be a Cohen--Macaulay local ring of dimension $d$.
Let $I$ be an ideal of $R$ such that $R/I$ is a maximal Cohen--Macaulay $R$-module.
Then the following statements hold.
\begin{enumerate}[\rm(1)]
\item
$\ind\cm(R/I)$ is contained in $\ind\cm(R)$.
\item
$\ind\cmp(R/I)$ is contained in $\ind\cmp(R)$.
\item
$\ind\cm(R/I)$ is contained in $\ind\cmp(R)$, if $\v(I)\subseteq\v(0:I)$.
\end{enumerate}
\end{prop}

\begin{proof}
Let $M$ be an indecomposable maximal Cohen--Macaulay $R/I$-module.
The definition of indecomposability says $M\ne0$.
The equalities $\depth M=\dim R/I=\dim R$ imply $M$ is a maximal Cohen--Macaulay $R$-module.
It is directly checked that $M$ is indecomposable as an $R$-module.
Now (1) follows.

Let $\p$ be a prime ideal of $R$ such that $M_\p\cong(R_\p)^{\oplus n}$ for some $n\ge0$.
If $n=0$, then $M_\p=0$.
If $n>0$, then $IR_\p=0$ since $IM=0$, and hence $M_\p\cong R_\p^{\oplus n}=(R/I)_\p^{\oplus n}$.

Let us consider the case where $M$ is in $\cmp(R/I)$.
Then there is a prime ideal $\q$ of $R$ with $I\subseteq\q\ne\m$ such that $M_\q$ is not $(R/I)_\q$-free.
Letting $\p:=\q$ in the above argument, we observe that $M_\q$ is not $R_\q$-free (note that the zero module is free).
Thus $M$ is in $\cmp(R)$, and (2) follows.

Next we consider the case where $M$ is in $\cmz(R)$.
As $\dim M=\dim R/I=d>0$, there is a nonmaximal prime ideal $\r$ of $R$ such that $M_\r\ne0$.
Letting $\p:=\r$ in the above argument, we have $IR_\r=0$.
Hence $\r$ is not in the support of the $R$-module $I$, which is equivalent to saying that $\r$ does not contain $(0:I)$.
On the other hand, $\r$ is in the support of the $R$-module $M$, which implies that $\r$ contains $I$.
Thus $\v(I)$ is not contained in $\v(0:I)$.
We now observe that (3) holds.
\end{proof}

The lemma below says finite $\cm$-representation type is equivalent to finite $\cmz$-representation type.

\begin{lem}\label{c2}
Let $R$ be a Cohen--Macaulay local ring.
If $R$ has infinite $\cm$-representation type, then $R$ has infinite $\cmz$-representation type.
\end{lem}

\begin{proof}
Suppose that $R$ has finite $\cmz$-representation type.
Then by \cite[Corollary 1.2]{dim} it is an isolated singularity.
Hence $\cm(R)=\cmz(R)$, and we have $\ind\cm(R/I)=\ind\cmz(R/I)$, which is a finite set.
This contradicts the assumption that $R$ has infinite $\cm$-representation type.
\end{proof}

Now we can prove the first main result of this section, which gives various necessary conditions for a Cohen--Macaulay local ring to have finite $\cmp$-representation type.

\begin{thm}\label{19}
Let $R$ be a Cohen--Macaulay local ring of dimension $d>0$.
Let $I$ be an ideal of $R$, and assume that $R/I$ is a maximal Cohen--Macaulay $R$-module.
Then $R$ has infinite $\cmp$-representation type in each of the following cases.
\begin{enumerate}[\rm(1)]
\item
$R/I$ has infinite $\cmp$-representation type.
\item
$\v(I)\subseteq\v(0:I)$ and
\begin{enumerate}[\rm(a)]
\item
$R/I$ has infinite $\cm$-representation type, or
\item
$d\ge2$.
\end{enumerate}
\item
$\height(I+(0:I))<d$, $R/I$ has infinite $\cm$-representation type, and
\begin{enumerate}[\rm(a)]
\item
$R/I$ is a Gorenstein ring, or
\item
$R/I$ is a domain, or
\item
$d=1$ and $R/I$ is analytically unramified, or
\item
$d=1$, $k$ is infinite, and $R/I$ is equicharacteristic and reduced.
\end{enumerate}
\end{enumerate}
\end{thm}

\begin{proof}
(1)\&(2a) These assertions immediately follow from (2) and (3) of Proposition \ref{10}, respectively.

(2b) In view of (2a), we may assume that $R/I$ has finite $\cm$-representation type.
It follows from \cite[Corollary 2]{HL02} that $R/I$ is an isolated singularity.
As $d\ge2$, the ring $R/I$ is a (normal) domain.
Hence $\p:=I$ is a prime ideal of $R$.
As $\dim R/\p=d$, the prime ideal $\p$ is minimal.
The assumption $\v(\p)\subseteq\v(0:\p)$ implies $(0:_R\p)\subseteq\p$.
Localizing this inclusion at $\p$, we get an inclusion $(0:_{R_\p}\p R_\p)\subseteq\p R_\p$, which particularly says that $R_\p$ is not a field.
Therefore $\p$ belongs to $\sing R$.

Suppose that $R$ has finite $\cmp$-representation type.
Then Corollary \ref{16}(1) implies that $\sing R$ has dimension at most one.
In particular, we obtain $d=\dim R/\p\le1$, which is a contradiction.
Consequently, $R$ has infinite $\cmp$-representation type.

(3) We find a nonmaximal prime ideal $\p$ of $R$ that contains the ideal $I+(0:I)$.
Then, as $\p$ contains $I$, the prime ideal $\p/I$ of $R/I$ is defined, which is not maximal.
Also, since $\p$ contains $(0:I)$ as well, we see that $IR_\p$ is a nonzero proper ideal of $R_\p$.

We establish several claims.

\begin{claim}\label{c1}
Let $M\in\ind\cmz(R/I)$ with $M_\p\ne0$.
Then $M\in\ind\cmp(R)$.
\end{claim}

\begin{proof}[Proof of Claim]
Proposition \ref{10}(1) implies $M\in\ind\cm(R)$.
There exists an integer $n\ge0$ such that
$$
M_\p=M_{\p/I}\cong(R/I)_{\p/I}^{\oplus n}=(R/I)_\p^{\oplus n}=(R_\p/IR_\p)^{\oplus n}.
$$
Since $M_\p$ is nonzero, we have to have $n>0$.
Since $IR_\p$ is a nonzero proper ideal of $R_\p$, we have that $M_\p$ is not a free $R_\p$-module.
We now conclude that $M$ belongs to $\ind\cmp(R)$.
\renewcommand{\qedsymbol}{$\square$}
\end{proof}

\begin{claim}\label{c3}
When $R/I$ is Gorenstein, for each $M\in\ind\cmz(R/I)$, either $M$ or $\syz_{R/I}M$ is in $\ind\cmp(R)$.
\end{claim}

\begin{proof}[Proof of Claim]
If $M_\p\ne0$, then $M\in\ind\cmp(R)$ by Claim \ref{c1}.
Assume $M_\p=0$.
There is an exact sequence $0 \to N \to (R/I)^{\oplus n} \to M \to 0$, where we set $N:=\syz_{R/I}M$ and $n:=\nu_{R/I}(M)>0$.
Localization at $\p$ gives an isomorphism $N_\p\cong(R_\p/IR_\p)^{\oplus n}$.
As $n>0$ and $IR_\p$ is a proper ideal, the module $N_\p$ is nonzero.
Since $R/I$ is Gorenstein, we apply Lemma \ref{8} and \cite[Lemma (8.17)]{Y} to see that $N$ belongs to $\ind\cmz(R/I)$.
Using Claim \ref{c1} again, we obtain $N\in\ind\cmp(R)$.
\renewcommand{\qedsymbol}{$\square$}
\end{proof}

\begin{claim}\label{ir}
There is an inclusion
$$
\{M\in\ind\cmz(R/I)\mid\text{$M$ has a rank as an $R/I$-module}\}\subseteq\ind\cmp(R).
$$
\end{claim}

\begin{proof}[Proof of Claim]
Take $M$ from the left-hand side.
Since the $R/I$-module $M$ is maximal Cohen--Macaulay, its annihilator has grade $0$.
Hence $M$ has positive rank, and we see that $\supp_{R/I}M=\spec R/I$.
Therefore $M_\p=M_{\p/I}$ is nonzero.
It follows from Claim \ref{c1} that $M$ belongs to $\ind\cmp(R)$.
\renewcommand{\qedsymbol}{$\square$}
\end{proof}

(3a) Suppose that $R$ has finite $\cmp$-representation type, namely, $\ind\cmp(R)$ is a finite set.
Lemma \ref{c2} guarantees that the set $\ind\cmz(R/I)$ is infinite, and hence the set difference
$$
\SS:=\ind\cmz(R/I)\setminus\ind\cmp(R)
$$
is infinite as well.
Thus we can choose a (countably) infinite subset $\{M_1,M_2,M_3,\dots\}$ of $\SS$.
By Claim \ref{c3} we see that $\syz_{R/I}M_i$ belongs to $\ind\cmp(R)$ for all $i$.
Note that $\syz_{R/I}M_i\not\cong\syz_{R/I}M_j$ for all distinct $i,j$ since $R/I$ is Gorenstein and $M_i,M_j$ are maximal Cohen--Macaulay over $R/I$.
It follows that $\ind\cmp(R)$ is an infinite set, which is a contradiction.
Thus $R$ has infinite $\cmp$-representation type.

(3b) Since $R/I$ is a domain, every $R/I$-module has a rank.
Claim \ref{ir} implies that $\ind\cmz(R/I)$ is contained in $\ind\cmp(R)$, while $\ind\cmz(R/I)$ is an infinite set by Lemma \ref{c2}.
It follows that $R$ has infinite $\cmp$-representation type.

(3c) Note that $\cm(R/I)=\cmz(R/I)$.
Since $R/I$ is analytically unramified and has infinite $\cm$-representation type, it follows from \cite[Theorem 4.10]{LW} that the left-hand side of the inclusion in Claim \ref{ir} is infinite, and so is the right-hand side $\ind\cmp(R)$, that is, $R$ has infinite $\cmp$-represenation type.

(3d) Since $k$ is infinite and $R/I$ is equicharacteristic, we can apply \cite[Theorem 17.10]{LW} to deduce that if $R/I$ has unbounded $\cm$-representation type, then the left-hand side of the inclusion in Claim \ref{ir} is infinite (as $R/I$ is reduced), and we are done.
Hence we may assume that $R/I$ has bounded $\cm$-representation type.
By \cite[Theorems 10.1 and 17.10]{LW} the completion $\widehat{R/I}$ has infinite and bounded $\cm$-representation type.
According to \cite[Theorem 17.9]{LW}, the ring $\widehat{R/I}$ is isomorphic to one of the following three rings.
$$
k[\![x,y]\!]/(x^2),\quad
k[\![x,y]\!]/(x^2y),\quad
k[\![x,y,z]\!]/(yz,x^2-xz,xz-z^2).
$$
The indecomposable maximal Cohen--Macaulay modules over these rings are classified; one can find complete lists of those modules in \cite[Propositions 4.1 and 4.2]{BGS} and \cite[Example 14.23]{LW}.
We can check by hand that each of these rings has an infinite family of nonisomorphic indecomposable maximal Cohen--Macaulay modules of rank $1$.
This family of modules is extended from a family of $R/I$-modules by \cite[Corollary 2.2]{LW2}, and these are nonisomorphic indecomposable maximal Cohen--Macaulay $R/I$-modules of rank $1$.
Again, the left-hand side of the inclusion in Claim \ref{ir} is infinite, and the proof is completed.
\end{proof}

Two irreducible elements $p,q$ of an integral domain $R$ are said to be {\em distinct} if $pR\ne qR$.
Applying our Theorem \ref{19}, we can obtain the following corollary, which is a basis in the next Section \ref{m} to obtain a stronger result (Theorem \ref{46}).

\begin{cor}\label{39}
Let $(S,\n)$ be a regular local ring of dimension two.
Take an element $0\ne f\in\n$ and set $R=S/(f)$.
Suppose that $R$ is not an isolated singularity $($equivaently, is not reduced$)$ but has finite $\cmp$-representation type.
Then $f$ has one of the following forms:
$$
f=
\begin{cases}
p^2qr & \text{where $p,q,r$ are distinct irreducibles with $S/(pqr)$ having finite $\cm$-representation type},\\
p^2q & \text{where $p\ne q$ are irreducibles with $S/(pq)$ having finite $\cm$-representation type},\\
p^2 & \text{where $p$ is an irreducible with $S/(p)$ having finite $\cm$-representation type}.
\end{cases}
$$
\end{cor}

\begin{proof}
As $S$ is factorial, we can write $f=p_1^{a_1}\cdots p_n^{a_n}$, where $p_1,\dots,p_n$ are distinct irreducible elements and $n,a_1,\dots,a_n$ are positive integers.
If $a_1=\cdots=a_n=1$, then $R$ is reduced, and hence it is an isolated singularity, which is a contradiction.
Thus we may assume $a_1\ge2$.

Put $x:=p_1\cdots p_n\in R$.
We have
$$
(x)+(0:x)=(p_1\cdots p_n,p_1^{a_1-1}p_2^{a_2-1}\cdots p_n^{a_n-1})\subseteq(p_1),
$$
and hence $\height((x)+(0:x))=0<1$.
Taking advantage of Theorem \ref{19}(3a), we observe that $R/(x)$ has finite $\cm$-representation type.
Also, $R/(x)=S/(p_1\cdots p_n)$ has multiplicity at least $n$.
By \cite[Theorem 4.2 and Proposition 4.3]{LW} we see that $n\le3$.

Assume either $a_1\ge3$ or $a_l\ge2$ for some $l\ge2$, say $l=2$.
Then put $x:=p_1^2p_2\cdots p_n\in R$.
We have
$$
(x)+(0:x)=(p_1^2p_2\cdots p_n,p_1^{a_1-2}p_2^{a_2-1}\cdots p_n^{a_n-1})\subseteq\begin{cases}
(p_1) & \text{(if $a_1\ge3$)},\\
(p_2) & \text{(if $a_2\ge2$)}
\end{cases}
$$
and hence $\height((x)+(0:x))=0<1$.
The ring $R/(x)=S/(p_1^2p_2\cdots p_n)$ is not reduced, so it is not an isolated singularity.
By \cite[Corollary 2]{HL02}, it has infinite $\cm$-representation type.
Theorem \ref{19}(3a) implies that $R$ has infinite $\cmp$-representation type, which is a contradiction.
Thus $a_1=2$ and $a_2=\cdots=a_n=1$.

Getting together all the above arguments completes the proof of the corollary.
\end{proof}

To give applications of Theorem \ref{19}, we establish a lemma.

\begin{lem}\label{41}
Let $R$ be a Gorenstein local ring of finite $\cmp$-representation type.
Then for all $M\in\ind\cmp(R)$ one has $\cx_RM=1$.
\end{lem}

\begin{proof}
As $R$ is Gorenstein, $\syz^iM\in\ind\cmp(R)$ for all $i\ge0$ by Lemma \ref{8} and \cite[Lemma 8.17]{Y}.
Since $\ind\cmp(R)$ is a finite set, $\syz^tM$ is periodic for some $t\ge0$.
Hence $M$ has complexity at most one.
As $M$ is in $\cmp(R)$, it has to have infinite projective dimension.
Thus the complexity of $M$ is equal to one.
\end{proof}

Let $R$ be a ring.
We denote by $\ds(R)$ the {\em singularity category} of $R$, that is, the Verdier quotient of the bounded derived category of finitely generated $R$-modules by perfect complexes.
For an $R$-module $M$, we denote by $\nf_R(M)$ the {\em nonfree locus} of $M$, that is, the set of prime ideals $\p$ of $R$ such that $M_\p$ is nonfree as an $R_\p$-module.
Now we prove the following result by using Theorem \ref{19}.

\begin{thm}\label{15}
Let $R$ be a Cohen--Macaulay local ring of dimension $d>0$.
Let $I$ be an ideal of $R$ with $\v(I)\subseteq\v(0:I)$, and assume that $R/I$ is a maximal Cohen--Macaulay $R$-module.
Suppose that $R$ has finite $\cmp$-representation type.
Then:
\begin{enumerate}[\rm(1)]
\item
One has $d=1$.
\item
If $I^n=0$, then $\dim\cm(R)\le n-1$.
\item
If $R$ is Gorenstein, then $R$ is a hypersurface and $\dim\ds(R)\le n-1$.
\end{enumerate}
\end{thm}

\begin{proof}
(1) This is a direct consequence of Theorem \ref{19}(2b).

(2) It follows from Theorem \ref{19}(2a) that $R/I$ has finite $\cm$-representation type.
Hence there exists a maximal Cohen--Macaulay $R/I$-module $G$ such that $\cm(R/I)=\add_{R/I}G$.
Take any maximal Cohen--Macaulay $R$-module $M$ and put $M_0:=M$.
For each integer $0\le i\le n-1$ we have an exact sequence $0 \to (0:_{M_i}I) \xrightarrow{f_i} M_i \to M_{i+1} \to 0$, where $f_i$ is the inclusion map.

Let us show that for all $0\le i\le n-1$ the $R$-module $M_i$ is maximal Cohen--Macaulay and annihilated by $I^{n-i}$.
We use induction on $i$.
It clearly holds in the case $i=0$, so let $i\ge1$.
Applying the functor $\Hom_R(-,M_{i-1})$ to the natural exact sequence $0\to I\to R\to R/I\to0$ induces an exact sequence $0 \to (0:_{M_{i-1}}I) \xrightarrow{f_{i-1}} M_{i-1} \to \Hom_R(I,M_{i-1})$, and hence $M_i$ is identified with a submodule of $\Hom_R(I,M_{i-1})$.
The induction hypothesis implies that $M_{i-1}$ is maximal Cohen--Macaulay and $I^{n-i-1}M_{i-1}=0$.
Then $\Hom_R(I,M_{i-1})$ has positive depth (see \cite[Exercise 1.4.19]{BH}), and so does $M_i$.
Since $d=1$ by (1), the $R$-module $M_i$ is maximal Cohen--Macaulay.
Also, $I^{n-i}M_{i-1}$ is contained in $(0:_{M_{i-1}}I)$, which implies that $I^{n-i}$ annihilates $M_{i-1}/(0:_{M_{i-1}}I)=M_i$.

Thus, for each $0\le i\le n-1$ the submodule $(0:_{M_i}I)$ of $M_i$ is also maximal Cohen--Macaulay (as $d=1$ again).
Since it is killed by $I$, it is a maximal Cohen--Macaulay $R/I$-module.
Therefore $(0:_{M_i}I)$ belongs to $\add_RG={[G]}_1$ for all $0\le i\le n-1$.
Using that fact that $M_0=M$ and $M_n=0$, we easily observe that $M$ belongs to ${[G]}_n$.
It is concluded that $\cm(R)={[G]}_n$, which means that $\dim\cm(R)\le n-1$.

(3) We claim that the $R$-module $R/I$ has complexity at most one.
Indeed, we have
$$
\nf_R(R/I)=\v(I+(0:I))=\v(I)\cap\v(0:I)=\v(I),
$$
where the first equality follows from \cite[Proposition 1.15(4)]{stcm}.
As $I$ is not $\m$-primary, $\nf_R(R/I)$ contains a nonmaximal prime ideal of $R$.
Hence $R/I$ is in $\cmp(R)$.
Since $R/I$ is a local ring, it is an indecomposable $R$-module, and therefore $R/I\in\ind\cmp(R)$.
It is seen from Lemma \ref{41} that $R/I$ has complexity at most one as an $R$-module.
Now the claim follows.

Let $X$ be an indecomposable $R/I$-module which is a direct summand of $C:=\syz_{R/I}^dk$.
Proposition \ref{10}(3) implies that $X$ belongs to $\ind\cmp(R)$.
As in the proof of the first claim, $\syz_R^iX$ belongs to $\ind\cmp(R)$ for all $i\ge0$, and $\syz_R^nX$ is periodic for some $n\ge0$.
Therefore, we find an integer $m\ge0$ such that $\syz_R^mC$ is periodic; see Lemma \ref{9}.
This implies that $C$ has complexity at most one.
There is an exact sequence
$$
0 \to C \to (R/I)^{\oplus r_{m-1}} \to \cdots \to (R/I)^{\oplus r_2} \to (R/I)^{\oplus r_1} \to R/I \to k\to0.
$$
As $\cx_RC\le1$ and $\cx_R(R/I)\le1$, we get $\cx_Rk\le1$.
By \cite[Theorem 8.1.2]{A} the ring $R$ is a hypersurface.
The last assertion follows from \cite[Theorem 4.4.1]{B} and \cite[Proposition 3.5(3)]{dim}.
\end{proof}

The above theorem gives rise to the two corollaries below.
Note that the theorem and the two corollaries all give answers to Questions \ref{2} and \ref{61}.

\begin{cor}\label{30}
Let $R$ be a Cohen--Macaulay local ring of dimension $d>0$ possessing an element $x\in R$ with $(0:x)=(x)$.
Suppose that $R$ has finite $\cmp$-representation type.
Then $d=1$ and $\dim\cm(R)\le1$.
If $R$ is Gorenstein, then $R$ is a hypersurface and $\dim\ds(R)\le1$.
\end{cor}

\begin{proof}
We have $x^2=0$.
The sequence $\cdots\xrightarrow{x}R\xrightarrow{x}R\xrightarrow{x}\cdots$ is exact, which implies that $R/(x)$ is a maximal Cohen--Macaulay $R$-module.
The assertions follow from Theorem \ref{15}.
\end{proof}

\begin{cor}\label{42}
Let $R$ be a Gorenstein non-reduced local ring of dimension one.
If $R$ has finite $\cmp$-representation type, then $R$ is a hypersurface.
\end{cor}

\begin{proof}
Since $R$ does not have an isolated singularity, $\sing R$ contains a nonmaximal prime ideal $\p$.
It is easy to see that $(R/\p)_\p=\kappa(\p)$ is not $R_\p$-free, and we also have $\v(\p)=\{\p,\m\}\subseteq\supp_R(\p)=\v(0:\p)$ as $\p R_\p\ne0$.
Lemma \ref{41} implies that the $R$-module $R/\p$ has complexity at most $1$, and the local ring $R$ is a hypersurface by virtue of Theorem \ref{15}(3).
\end{proof}

\section{The one-dimensional hypersurfaces of finite CM$_+$-representation type}\label{m}

The purpose of this section is to prove the following theorem.

\begin{thm}\label{46}
Let $R$ be a homomorphic image of a regular local ring.
Suppose that $R$ does not have an isolated singularity but is Gorenstein.
If $\dim R=1$, then the following are equivalent.
\begin{enumerate}[\rm(1)]
\item
The ring $R$ has finite $\cmp$-representation type.
\item
There exist a regular local ring $S$ and a regular system of parameters $x,y$ such that $R$ is isomorphic to $S/(x^2)$ or $S/(x^2y)$.
\end{enumerate}
When either of these two conditions holds, the ring $R$ has countable $\cm$-representation type.
\end{thm}

In fact, the last assertion and the implication $(2)\Rightarrow(1)$ follow from \cite[Propositions 4.1 and 4.2]{BGS} and \cite[Proposition 2.1]{hsccm}, respectively.
The implication $(1)\Rightarrow(2)$ is an immediate consequence of the combination of Corollaries \ref{39}, \ref{42} with Theorems \ref{32}, \ref{33}, \ref{34} shown in this section.
Note by Theorem \ref{5} that the above theorem guarantees that under the assumption that $R$ is a complete Gorenstein local ring of dimension one, Question \ref{61} has an affirmative answer.

We establish three subsections, whose purposes are to prove Theorems \ref{32}, \ref{33} and \ref{34}, respectively.

\subsection{The hypersurface $S/(p^2)$}

For a ring $A$ we denote by $\nzd(A)$ the set of non-zerodivisors of $A$, and by $\Q(A)$ the total quotient ring $A_{\nzd(A)}$ of $A$.
A ring extension $A\subseteq B$ is called {\em birational} if $B\subseteq\Q(A)$.

\begin{lem}\label{31}
Let $A\subseteq B$ be a birational extension.
Let $M$ be a $B$-module which is torsion-free as an $A$-module.
If $M$ is indecomposable as a $B$-module, then $M$ is indecomposable as an $A$-module as well.
\end{lem}

\begin{proof}
From the proof of \cite[Proposition 4.14]{LW}, we have $\End_A(M)=\End_B(M)$.  The claim then follows from from \cite[Proposition 1.1]{LW}.
\end{proof}

Let $A$ be a ring and $M$ an $A$-module.
We denote by $\lend_A(M)$ the quotient of $\End_A(M)$ by the endomorphisms factoring through projective $A$-modules.
For a flat $A$-algebra $B$ one has $\lend_A(M)\otimes_AB\cong\lend_B(M\otimes_AB)$; this can be shown by using \cite[Lemma 3.9]{Y}.

\begin{lem}\label{29}
Let $A\subseteq B$ be a finite birational extension of $1$-dimensional Cohen--Macaulay local rings.
Then $\ind\cmp(B)$ is contained in $\ind\cmp(A)$.
\end{lem}

\begin{proof}
Let $M\in\ind\cmp(B)$.
Then $\depth_AM=\depth_BM>0$, which shows that $M$ is maximal Cohen--Macaulay as an $A$-module.
Lemma \ref{31} implies $M\in\ind\cm(A)$.
Set $Q=\Q(A)=\Q(B)$.
Applying the functor $Q\otimes_A-$ to the inclusions $A\subseteq B\subseteq Q$ yields $B\otimes_AQ=Q$.
Hence we have
$$
M\otimes_BQ=M\otimes_B(B\otimes_AQ)=M\otimes_AQ,\quad
\lend_A(M)\otimes_AQ
\cong\lend_Q(M\otimes_AQ)
\cong\lend_Q(M\otimes_BQ).
$$
Since $M$ is in $\cmp(B)$, there is a minimal prime $P$ of $B$ such that $M_P$ is not $B_P$-free.
Note that $M_P=(M\otimes_BQ)\otimes_QQ_P$ and $Q_P=B_P$.
Hence $M\otimes_BQ$ is not $Q$-projective, and we obtain $\lend_Q(M\otimes_BQ)\ne0$.
Therefore $\lend_A(M)\otimes_AQ$ is nonzero, which means that the $A$-module $\lend_A(M)$ is not torsion.
Thus $\supp_A(\lend_A(M))$ contains a minimal prime of $A$, which implies that $M$ belongs to $\cmp(A)$.
Consequently, we obtain $M\in\ind\cmp(A)$, and the lemma follows.
\end{proof}

The following lemma is a consequence of \cite[Corollary 7.6]{Y}, which is used not only now but also later.

\begin{lem}\label{20}
Let $(S,\n)$ be a regular local ring and $x\in\n$, and set $R=S/(x)$.
Then
$$
\{M\in\cm(R)\mid\text{$M$ is cyclic}\}/_{\cong}=
\{R/yR\mid\text{$y\in S$ with $x\in yS$}\}/_{\cong}.
$$
In particular, there exist only finitely many nonisomorphic indecomposable cyclic maximal Cohen--Macaulay $R$-modules.
\end{lem}

Now we can achieve the purpose of this subsection.

\begin{thm}\label{32}
Let $(S,\n)$ be a regular local ring of dimension two, and let $p\in\n^2$ be an irreducible element.
Then $R=S/(p^2)$ has infinite $\cmp$-representation type.
\end{thm}

\begin{proof}
Take any element $t\in\n$ that is regular on $R$.
We consider the $S$-algebra $T=S[z]/(tz-p,z^2)$, where $z$ is an indeterminate over $S$.
We establish two claims.

\setcounter{claim}{0}
\begin{claim}\label{27}
The ring $T$ is a local complete intersection of dimension $1$ and codimension $2$ with $t$ being a system of parameters.
\end{claim}

\begin{proof}[Proof of Claim]
It is clear that $T=S[\![z]\!]/(tz-p,z^2)$, which shows that $T$ is a local ring, and $\dim T=\dim S[\![z]\!]-\height(tz-p,z^2)\ge3-2=1$ by Krull's Hauptidealsatz.
We have $T/tT=S[\![z]\!]/(t,p,z^2)=(S/(t,p))[\![z]\!]/(z^2)$.
As $S/(t,p)$ is artinian, so is $T/tT$.
Hence $\dim T=1$ and $t$ is a system of parameters of $T$, and thus $T$ is a complete intersection (the equalities $\dim S[\![z]\!]=3$ and $\dim T=1$ imply $\height(tz-p,z^2)=2$, whence $tz-p,z^2$ is a regular sequence).
As $(tz-p,z^2)\subseteq\n^2$, the local ring $T$ has codimension $2$.
\renewcommand{\qedsymbol}{$\square$}
\end{proof}

\begin{claim}\label{28}
The ring $R$ is naturally embedded in $T$, and this embedding is a finite birational extension.
\end{claim}

\begin{proof}[Proof of Claim]
Let $\phi:S\to T$ be the natural map and put $I=\ker\phi$.
As $p^2=t^2z^2=0$ in $T$, we have $(p^2)\subseteq I$.
Hence the map $\phi$ factors as $S\twoheadrightarrow R\twoheadrightarrow S/I\hookrightarrow T$.
It is seen that $T$ is an $R$-module generated by $1,z$ and $S/I$ is an $R$-submodule of $T$.
Since $T$ has positive depth by Claim \ref{27}, so does $S/I$.
Thus $S/I$ is a maximal Cohen--Macaulay cyclic module over the hypersurface $R$, and Lemma \ref{20} implies that $I$ coincides with either $(p)$ or $(p^2)$.
If $I=(p)$, then $T=T/pT=S[z]/(tz,p,z^2)$, which contradicts the fact following from Claim \ref{27} that $t$ is $T$-regular.
We get $I=(p^2)$, which means the map $R\to T$ is injective.

Let $C$ be the cokernel of the injection $R\hookrightarrow T$.
Then $C$ is generated by $z$ as an $R$-module.
Note that $tz=p=0$ in $C$.
Hence $C$ is a torsion $R$-module, which means $C\otimes_R\Q(R)=0$.
Thus $(\Q(R)\to T\otimes_R\Q(R))=(R\hookrightarrow T)\otimes_R\Q(R)$ is an isomorphism, while the natural map $T\to T\otimes_R\Q(R)$ is injective as $T$ is maximal Cohen--Macaulay over $R$ by Claim \ref{27}.
Thus the embedding $R\hookrightarrow T$ is birational.
\renewcommand{\qedsymbol}{$\square$}
\end{proof}

By Claim \ref{27}, the ring $T$ is a complete intersection, which implies that the element $z^2$ is regular on the ring $S[z]/(tz-p)$ and so is $z$.
It is easy to check that $(0:_Tz)=zT$.
Claim \ref{27} also guarantees that $T$ is not a hypersurface.
It follows from Corollary \ref{30} that $T$ has infinite $\cmp$-representation type.
Combining Claim \ref{28} with Lemma \ref{29}, we obtain the inclusion $\ind\cmp(T)\subseteq\ind\cmp(R)$.
We now conclude that $R$ has infinite $\cmp$-representation type, and the proof of the theorem is completed.
\end{proof}

\subsection{The hypersurface $S/(p^2qr)$}

\begin{setup}
Throughout this subsection, let $(S,\n)$ be a $2$-dimensional regular local ring and $p,q,r$ pairwise distinct irreducible elements of $S$.
Let $R=S/(p^2qr)$ be a local hypersurface of dimension $1$.
Setting $\p=pR$, $\q=qR$, $\r=rR$ and $\m=\n R$, one has $\spec R=\{\p,\q,\r,\m\}$.
For each $i\ge1$ we define matrices
$$
A_i=
\begin{pmatrix}
p&0&r^i\\
0&pq&p\\
0&0&pr
\end{pmatrix},\qquad
B_i=
\begin{pmatrix}
pqr&0&-qr^i\\
0&pr&-p\\
0&0&pq
\end{pmatrix}
$$
over $S$.
Put $M_i=\cok_SA_i$ and $N_i=\cok_SB_i$.
\end{setup}

\begin{lem}\label{23}
\begin{enumerate}[\rm(1)]
\item
For every $i\ge1$ it holds that $M_i,N_i\in\cmp(R)$, $\syz_RM_i=N_i$ and $\syz_RN_i=M_i$.
\item
For all positive integers $i\ne j$, one has $M_i\ncong M_j$ and $N_i\ncong N_j$ as $R$-modules.
\end{enumerate}
\end{lem}

\begin{proof}
(1) It is clear that $A_iB_i=B_iA_i=p^2qrE$.
Hence $A_i,B_i$ give a matrix factorization of $p^2qr$ over $S$, and we have $M_i,N_i\in\cm(R)$, $\syz_RM_i=N_i$ and $\syz_RN_i=M_i$; see \cite[Chapter 7]{Y}.
Note that $q,r$ are units and $p^2=0$ in $R_\p=S_{(p)}/p^2S_{(p)}$.
There are isomorphisms
$$
(M_i)_\p
\cong\cok\left(\begin{smallmatrix}
p&0&r^i\\
0&p&p\\
0&0&p
\end{smallmatrix}\right)\cong\cok\left(\begin{smallmatrix}
p&0&1\\
0&p&0\\
0&0&p
\end{smallmatrix}\right)\cong\cok\left(\begin{smallmatrix}
0&0&1\\
0&p&0\\
-p^2&0&p
\end{smallmatrix}\right)\cong\cok\left(\begin{smallmatrix}
0&0&1\\
0&p&0\\
0&0&0
\end{smallmatrix}\right)
\cong\cok\left(\begin{smallmatrix}
p&0\\
0&0
\end{smallmatrix}\right)\cong R_\p\oplus\kappa(\p),
$$
where all the cokernels are over $R_\p$.
Therefore $M_i\in\cmp(R)$, and we get $N_i\in\cmp(R)$ by Lemma \ref{8}.

(2) Suppose that there is an $R$-isomorphism $M_i\cong M_j$.
It then holds that $\fitt_2(M_i)=\fitt_2(M_j)$, which means $(p,r^i)R=(p,r^j)R$.
This implies that $(\overline{r}^i)=(\overline{r}^j)$ in the integral domain $R/\p=S/(p)$.
Since $\overline{r}\ne\overline{0}$ in this ring, we get $i=j$.
If $N_i\cong N_j$, then $M_i\cong\syz_RN_i\cong\syz_RN_j\cong M_j$ by (1), and we get $i=j$.
\end{proof}

\begin{lem}\label{26}
There is an equality
$$
\{M\in\cmp(R)\mid\text{$M$ is cyclic}\}/_{\cong}=\{R/(p),\,R/(pq),\,R/(pr),\,R/(pqr)\}/_{\cong}.
$$
\end{lem}

\begin{proof}
Let $M$ be a cyclic $R$-module with $M\in\cmp(R)$.
It follows from Lemma \ref{20} that $M$ is isomorphic to $R/fR$ for some element $f\in S$ which divides $p^2qr$ in $S$.
The localizations $R_\q,R_\r$ are fields, and hence $M_\p$ is not $R_\p$-free.
As $p^2=0$ in $R_\p=S_{(p)}/p^2S_{(p)}$, it is observed that $f\in pS\setminus p^2S$.
Thus $f\in\{p,pq,pr,pqr\}$.
Conversely, for any $g\in\{p,pq,pr,pqr\}$ we have $(R/gR)_\p\cong\kappa(\p)$ and get $R/gR\in\cmp(R)$.
\end{proof}

\begin{lem}\label{25}
Let $i\ge1$ be an integer.
Then neither $\cok_{S/(pq)}\left(\begin{smallmatrix}
p&r^i\\
0&p
\end{smallmatrix}\right)$ nor $\cok_{S/(pr)}\left(\begin{smallmatrix}
p\\
qr^i
\end{smallmatrix}\right)$ contains $S/(p)$ as a direct summand.
\end{lem}

\begin{proof}
(1) Set $T=S/(pq)$ and $C=\cok_T\left(\begin{smallmatrix}
p&r^i\\
0&p
\end{smallmatrix}\right)$.
Consider the sequence
$$
T^{\oplus2}\xleftarrow{\left(\begin{smallmatrix}p&r^i\\0&p\end{smallmatrix}\right)}T^{\oplus2}\xleftarrow{\left(\begin{smallmatrix}q\\0\end{smallmatrix}\right)}T
$$
of homomorphisms of free $T$-modules.
Clearly, this is a complex.
Let $\left(\begin{smallmatrix}a\\b\end{smallmatrix}\right)\in T^{\oplus2}$ be such that $\left(\begin{smallmatrix}p&r^i\\0&p\end{smallmatrix}\right)\left(\begin{smallmatrix}a\\b\end{smallmatrix}\right)=\left(\begin{smallmatrix}0\\0\end{smallmatrix}\right)$.
In $S$ we have $pa+r^ib=pqc$ and $pb=pqd$ for some $c,d\in S$, and get $b=qd$ and $pa+r^iqd=pqc$.
Hence $pa\in qS\in\spec S$ and $a\in qS$; we find $e\in S$ with $a=qe$.
Then $pqe+r^iqd=pqc$, and $pe+r^id=pc$.
Therefore $r^id\in pS\in\spec S$, and $d\in pS$; we find $f\in S$ with $d=pf$ and get $b=qpf$.
In $T^{\oplus2}$ we have $\left(\begin{smallmatrix}a\\b\end{smallmatrix}\right)=\left(\begin{smallmatrix}qe\\pqf\end{smallmatrix}\right)=\left(\begin{smallmatrix}qe\\0\end{smallmatrix}\right)=\left(\begin{smallmatrix}q\\0\end{smallmatrix}\right)(e)$.
It follows that the above sequence is exact, and the sequence
$$
\cdots\xrightarrow{p}T\xrightarrow{q}T\xrightarrow{p}T\xrightarrow{\left(\begin{smallmatrix}q\\0\end{smallmatrix}\right)}T^{\oplus2}\xrightarrow{\left(\begin{smallmatrix}p&r^i\\0&p\end{smallmatrix}\right)}T^{\oplus2}\to C\to0
$$
gives a minimal free resolution of the $T$-module $C$.

Now, assume that $S/(p)=T/pT$ is a direct summand of $C$.
Then $C\cong T/pT\oplus T/I$ for some ideal $I$ of $T$.
There are equalities of Betti numbers
$$
2=\beta_1^T(C)=\beta_1^T(T/pT\oplus T/I)=\beta_1^T(T/pT)+\beta_1^T(T/I)=1+\beta_1^T(T/I),
$$
and we get $\beta_1^T(T/I)=1$.
This means $I$ is a nonzero proper principal ideal of $T$; we write $I=gT$ where $g$ is a nonzero nonunit of $T$.
The uniqueness of a minimal free resolution yields a commutative diagram
$$
\xymatrix{
\cdots\ar[r]^{p} & T\ar[d]_\cong^{u_3}\ar[r]^{q} & T\ar[d]_\cong^{u_2}\ar[r]^{p} & T\ar[d]_\cong^{u_1}\ar[r]^-{\left(\begin{smallmatrix}q\\0\end{smallmatrix}\right)} & T^{\oplus2}\ar[d]_\cong^{\left(\begin{smallmatrix}t_1&t_2\\t_3&t_4\end{smallmatrix}\right)=:t}\ar[rr]^{\left(\begin{smallmatrix}p&0\\0&g\end{smallmatrix}\right)} && T^{\oplus2}\ar[d]_\cong^{\left(\begin{smallmatrix}s_1&s_2\\s_3&s_4\end{smallmatrix}\right)=:s}\ar[rr] && C\ar@{=}[d]\ar[r] & 0\\
\cdots\ar[r]_{p} & T\ar[r]_{q} & T\ar[r]_{p} & T\ar[r]_-{\left(\begin{smallmatrix}q\\0\end{smallmatrix}\right)} & T^{\oplus2}\ar[rr]_{\left(\begin{smallmatrix}p&r^i\\0&p\end{smallmatrix}\right)} && T^{\oplus2}\ar[rr] && C\ar[r] & 0
}
$$
whose vertical maps are isomorphisms.
As $s,t$ are isomorphisms, their determinants $s_1s_4-s_2s_3$ and $t_1t_4-t_2t_3$ are units of $T$.
The commutativity of the diagram shows $s_3p=pt_3$ and $t_3q=0$ in $T$, which imply $s_3-t_3\in(0:_Tp)=qT$ and $t_3\in(0:_Tq)=pT$.
Hence $s_3$ is a nonunit of $T$, and therefore $s_1,s_4$ are units of $T$.
Again from the commutativity of the diagram we get $s_4g=pt_4$ and $s_2g=pt_2+r^it_4$ in $T$, which give $p(s_2s_4^{-1}t_4-t_2)=r^it_4$.
Hence $r^it_4\in pT\in\spec T$ and $t_4\in pT$.
We now get $t_1t_4-t_2t_3$ is in $pT$, which contradicts the fact that it is a unit of $T$.
Consequently, $S/(p)$ is not a direct summand of $C$.

(2) Put $T=S/(pr)$ and $C=\cok_T\left(\begin{smallmatrix}
p\\
qr^i
\end{smallmatrix}\right)$.
We have $\spec T=\{pT,rT,\n T\}$.
Since $(p,qr^i)T$ is not contained in $pT$ or $rT$, it is $\n T$-primary and has positive grade.
Hence the sequence
$$
0 \to T \xrightarrow{\left(\begin{smallmatrix}
p\\
qr^i
\end{smallmatrix}\right)} T^{\oplus2} \to C\to0
$$
is exact, which gives a minimal free resolution of the $T$-module $C$.
This implies $\pd_RC=1$.

Suppose that $S/(p)=T/pT$ is a direct summand of $C$.
Then $T/pT$ has projective dimension at most one, which contradicts the fact that its minimal free resolution is $\cdots\xrightarrow{p}T\xrightarrow{q}T\xrightarrow{p}T\to T/pT\to0$.
It follows that $S/(p)$ is not a direct summand of $C$.
\end{proof}

\begin{lem}\label{24}
\begin{enumerate}[\rm(1)]
\item
The ring $S/(p,q)$ is artinian, and hence the number $\ell\ell(S/(p,q))$ is finite.
\item
Let $n\ge\ell\ell(S/(p,q))$ be a positive integer.
\begin{enumerate}[\rm(i)]
\item
If $X\in\cmp(R)$ is a cyclic direct summand of $M_n$, then $X$ is isomorphic to $R/(pqr)$.
\item
If $Y\in\cmp(R)$ is a cyclic direct summand of $N_n$, then $Y$ is isomorphic to $R/(pqr)$.
\end{enumerate}
\end{enumerate}
\end{lem}

\begin{proof}
(1) The factoriality of $S$ shows that $pS$ is a prime ideal of $S$.
As $pS\ne qS$, we have $\height(p,q)S>\height pS=1$.
Since $S$ has dimension two, the ideal $(p,q)S$ is $\n$-primary.
Thus $S/(p,q)S$ is an artinian ring.

(2i) There is an $R$-module $Z$ such that $M_n\cong X\oplus Z$.
According to Lemma \ref{26}, it holds that $X\cong R/(f)$ for some $f\in\{p,pq,pr,pqr\}$.
There are isomorphisms
$$
R/(f,r)\oplus Z/rZ
\cong M_n/rM_n
\cong\cok_{R/(r)}\left(\begin{smallmatrix}
p&0&0\\
0&pq&p\\
0&0&0
\end{smallmatrix}\right)
\cong\cok_{R/(r)}\left(\begin{smallmatrix}
p&0&0\\
0&0&p\\
0&0&0
\end{smallmatrix}\right)
\cong(R/(p,r))^{\oplus2}\oplus R/(r).
$$
Taking the completions and using the Krull--Schmidt property and \cite[Exercise 7.5]{E}, we observe that the ideal $(f,r)R$ coincides with either $(p,r)R$ or $rR$.
Hence $f\ne pq$.
Similarly, there are isomorphisms
$$
R/(f,q)\oplus Z/qZ
\cong M_n/qM_n
\cong\cok_{R/(q)}\left(\begin{smallmatrix}
p&0&r^n\\
0&0&p\\
0&0&pr
\end{smallmatrix}\right)
\cong\cok_{R/(q)}\left(\begin{smallmatrix}
p&0&r^n\\
0&0&p\\
0&0&0
\end{smallmatrix}\right)
\cong R/(q)\oplus\cok_{R/(q)}\left(\begin{smallmatrix}
p&r^n\\
0&p
\end{smallmatrix}\right).
$$
The assumption $n\ge\ell\ell(S/(p,q))$ implies $r^n\in\n^n\subseteq(p,q)S$.
We observe from this that $\cok_{R/(q)}\left(\begin{smallmatrix}
p&r^n\\
0&p
\end{smallmatrix}\right)\cong\cok_{R/(q)}\left(\begin{smallmatrix}
p&0\\
0&p
\end{smallmatrix}\right)$, and obtain an isomorphism $R/(f,q)\oplus Z/qZ\cong R/(q)\oplus(R/(p,q))^{\oplus2}$.
It follows that $(f,q)R$ coincides with either $qR$ or $(p,q)R$, which implies $f\ne pr$.
Finally, consider the isomorphisms
$$
R/(f,pq)\oplus Z/pqZ
\cong M_n/pqM_n
\cong\cok_{R/(pq)}\left(\begin{smallmatrix}
p&0&r^n\\
0&0&p\\
0&0&pr
\end{smallmatrix}\right)
\cong\cok_{R/(pq)}\left(\begin{smallmatrix}
p&0&r^n\\
0&0&p\\
0&0&0
\end{smallmatrix}\right)
\cong R/(pq)\oplus\cok_{R/(pq)}\left(\begin{smallmatrix}
p&r^n\\
0&p
\end{smallmatrix}\right).
$$
If $f=p$, then $R/(f,pq)=R/(p)$ and we see that this is a direct summand of $\cok_{R/(pq)}\left(\begin{smallmatrix}
p&r^n\\
0&p
\end{smallmatrix}\right)$, which contradicts Lemma \ref{25}.
Thus $f\ne p$, and we conclude that $f=pqr$.

(2ii) We go along the same lines as the proof of (2i).
We have $N_n\cong Y\oplus Z$ for some $Z\in\mod R$, and get $Y\cong R/(f)$ for some $f\in\{p,pq,pr,pqr\}$ by Lemma \ref{26}.
The isomorphisms
\begin{align*}
&R/(f,r)\oplus Z/rZ
\cong N_n/rN_n
\cong\cok_{R/(r)}\left(\begin{smallmatrix}
0&0&0\\
0&0&p\\
0&0&pq
\end{smallmatrix}\right)
\cong\cok_{R/(r)}\left(\begin{smallmatrix}
0&0&0\\
0&0&p\\
0&0&0
\end{smallmatrix}\right)
\cong R/(p,r)\oplus(R/(r))^{\oplus2},\\
&R/(f,q)\oplus Z/qZ
\cong N_n/qN_n
\cong\cok_{R/(q)}\left(\begin{smallmatrix}
0&0&0\\
0&pr&p\\
0&0&0
\end{smallmatrix}\right)
\cong\cok_{R/(q)}\left(\begin{smallmatrix}
0&0&0\\
0&0&p\\
0&0&0
\end{smallmatrix}\right)
\cong R/(p,q)\oplus(R/(q))^{\oplus2}
\end{align*}
show that $(f,q)$ (resp. $(f,r)$) coincides with either $(p,q)$ or $(q)$ (resp. either $(p,r)$ or $(r)$), which implies $f\ne pq,pr$.
We also have isomorphisms
$$
R/(f,pr)\oplus Z/prZ
\cong N_n/prN_n
\cong\cok_{R/(pr)}\left(\begin{smallmatrix}
0&0&qr^n\\
0&0&p\\
0&0&pq
\end{smallmatrix}\right)
\cong\cok_{R/(pr)}\left(\begin{smallmatrix}
0&0&qr^n\\
0&0&p\\
0&0&0
\end{smallmatrix}\right)
\cong\cok_{R/(pr)}\left(\begin{smallmatrix}
p\\
qr^n
\end{smallmatrix}\right)\oplus R/(pr).
$$
Using Lemma \ref{25}, we see that $f\ne p$, and obtain $f=pqr$.
\end{proof}

The purpose of this subsection is now fulfilled.

\begin{thm}\label{33}
Let $S$ be a regular local ring of dimension two.
Let $p,q,r$ be distinct irreducible elements of $S$.
Then $R=S/(p^2qr)$ has infinite $\cmp$-representation type.
\end{thm}

\begin{proof}
We assume that $R$ has finite $\cmp$-representation type, and derive a contradiction.
It follows from Lemma \ref{23}(1) that there exists an integer $a\ge1$ such that both $M_i$ and $N_i$ are decomposable for all $i\ge a$; we write $M_i\cong X_i\oplus Y_i$ for some $R$-modules $X_i,Y_i$ with $\nu(X_i)=1$ and $\nu(Y_i)=2$.
In view of Lemmas \ref{20} and \ref{23}(2), we see that there exists an integer $b\ge a$ such that $Y_h$ is indecomposable for all $h\ge b$ and that $Y_i\ncong Y_j$ for all $i,j\ge b$ with $i\ne j$.
Then, we have to have $Y_i\in\cmz(R)$ for all $i\ge b$, and hence $X_i\in\cmp(R)$ for all $i\ge b$ (by Lemma \ref{23}(1)).
Putting $c=\max\{b,\ell\ell(S/(p,q))\}$ and applying Lemma \ref{24}(2i), we obtain that $X_i$ is isomorphic to $R/(pqr)$ for all $i\ge c$.
There are isomorphisms
$$
N_i\cong\syz_RM_i\cong\syz_RX_i\oplus\syz_RY_i\cong\syz_R(R/(pqr))\oplus\syz_RY_i\cong R/(p)\oplus\syz_RY_i,
$$
where the first isomorphism follows from Lemma \ref{23}(1).
Since $R/(p)$ is in $\cmp(R)$, it follows from Lemma \ref{24}(2ii) that $R/(p)\cong R/(pqr)$, which is absurd.
\end{proof}

\subsection{The hypersurface $S/(p^2q)$}

The goal of this subsection is to prove the following theorem.

\begin{thm}\label{34}
Let $(S,\n)$ be a $2$-dimensional regular local ring.
Let $p,q$ be distinct irreducible elements of $S$.
Suppose that $R=S/(p^2q)$ has finite $\cmp$-representation type.
Then $p,q\notin\n^2$.
\end{thm}

Note that the rings $R$ and $R/p^2R$ are local hypersurfaces of dimension one.
If $p\in\n^2$, then $R/p^2R=S/(p^2)$ has infinite $\cmp$-representation type by Theorem \ref{32}, and so does $R$ by Theorem \ref{19}(1), which contradicts the assumption of the theorem.
Hence $p\notin\n^2$, and $p$ is a member of a regular system of parameters of $S$.
Thus we establish the following setting.

\begin{setup}
Throughout the remainder of this subsection, let $(S,\n)$ be a regular local ring of dimension two.
Let $x,y$ be a regular system of parameters of $S$, namely, $\n=(x,y)$.
Let $h\in\n^2$ be an irreducible element, and write $h=x^2s+xyt+y^2u$ with $s,t,u\in S$.
Let $R=S/(x^2h)$ be a local hypersurface of dimension one.
One has $\spec R=\{\p,\q,\m\}$, where we set $\p=xR$, $\q=hR$ and $\m=\n R$.
For each integer $i\ge1$ we define matrices
$$
A_i=
\begin{pmatrix}
x&0&y^i\\
0&xy&x\\
0&xh&0
\end{pmatrix},\qquad
B_i=
\begin{pmatrix}
xh&-y^ih&y^{i+1}\\
0&0&x\\
0&xh&-xy
\end{pmatrix}
$$
over $S$.
We put $M_i=\cok_SA_i$ and $N_i=\cok_SB_i$.
\end{setup}

In what follows, we argue along similar lines as in the previous subsection.

\begin{lem}[cf. Lemma \ref{23}]\label{37}
\begin{enumerate}[\rm(1)]
\item
Let $i\ge1$ be an integer.
The modules $M_i$ and $N_i$ belong to $\cmp(R)$, and it holds that $\syz_RM_i=N_i$ and $\syz_RN_i=M_i$.
\item
Let $i,j\ge1$ be integers with $i\ne j$.
One then have $M_i\ncong M_j$ and $N_i\ncong N_j$ as $R$-modules.
\end{enumerate}
\end{lem}

\begin{proof}
(1) We have $A_iB_i=B_iA_i=x^2hE$.
The matrices $A_i,B_i$ give a matrix factorization of $x^2h$ over $S$.
We have that $M_i,N_i$ are maximal Cohen--Macaulay $R$-modules with $\syz_RM_i=N_i$ and $\syz_RN_i=M_i$.
Note that $y,h$ are units and $x^2=0$ in $R_\p=S_{(x)}/x^2S_{(x)}$.
We have
\begin{align*}
(M_i)_\p
&\cong\cok_{R_\p}\left(\begin{smallmatrix}x&0&y^i\\0&xy&x\\0&x&0\end{smallmatrix}\right)
\cong\cok_{R_\p}\left(\begin{smallmatrix}x&0&y^i\\0&0&x\\0&x&0\end{smallmatrix}\right)
\cong\cok_{R_\p}\left(\begin{smallmatrix}x&0&1\\0&0&x\\0&x&0\end{smallmatrix}\right)
\cong\cok_{R_\p}\left(\begin{smallmatrix}0&0&1\\-x^2&0&x\\0&x&0\end{smallmatrix}\right)\\
&=\cok_{R_\p}\left(\begin{smallmatrix}0&0&1\\0&0&x\\0&x&0\end{smallmatrix}\right)
\cong\cok_{R_\p}\left(\begin{smallmatrix}0&0&1\\0&0&0\\0&x&0\end{smallmatrix}\right)
\cong\cok_{R_\p}\left(\begin{smallmatrix}0&0\\0&x\end{smallmatrix}\right)
\cong R_\p\oplus\kappa(\p),
\end{align*}
which shows that $M_i\in\cmp(R)$, and Lemma \ref{8} implies $N_i\in\cmp(R)$ as well.

(2) If $M_i\cong M_j$, then $(x,y^i)R=\fitt_2(M_i)=\fitt_2(M_j)=(x,y^j)R$, and $(\overline{y}^i)=(\overline{y}^j)$ in the discrete valuation ring $R/xR=S/(x)$ with $\overline y$ a uniformizer, which implies $i=j$.
As $N_i,N_j$ are the first syzygies of $M_i,M_j$ by (1), we see that if $N_i\cong N_j$, then $i=j$.
\end{proof}

\begin{lem}[cf. Lemma \ref{26}]\label{36}
It holds that $\{M\in\cmp(R)\mid\text{$M$ is cyclic}\}/_{\cong}=\{R/(x),\,R/(xh)\}/_{\cong}$.
\end{lem}

\begin{proof}
It is easy to see that neither $(R/(x))_\p$ nor $(R/(xh))_\p$ is $R_\p$-free.
Let $M\in\cmp(R)$ be cyclic.
As $R_\q$ is a field, $M_\p$ is not $R_\p$-free.
Using Lemma \ref{20}, we get $M\cong R/fR$ for some $f\in S$ with $f\mid x^2h$, $x\mid f$ and $x^2\nmid f$.
Hence, either $f=x$ or $f=xh$ holds.
\end{proof}

\begin{lem}[cf. Lemma \ref{24}]\label{38}
Let $i\ge1$ be an integer.
Let $C$ be a cyclic $R$-module with $C\in\cmp(R)$.
If $C$ is a direct summand of either $M_i$ or $N_i$, then $C$ is isomorphic to $R/(xh)$.
\end{lem}

\begin{proof}
(1) First, consider the case where $C$ is a direct summand of $M_i$.
Assume that $C$ is not isomorphic to $R/(xh)$.
Then $C\cong R/(x)$ by Lemma \ref{36}.
Application of the functor $-\otimes_RR/(xy)$ shows that $C/xyC=C\cong R/(x)$ is a direct summand of
$$
M_i/xyM_i
=\cok_{R/(xy)}\left(\begin{smallmatrix}x&0&y^i\\0&0&x\\0&xh&0\end{smallmatrix}\right)
\cong\cok_{R/(xy)}\left(\begin{smallmatrix}x&y^i&0\\0&x&0\\0&0&xh\end{smallmatrix}\right)
\cong\cok_{R/(xy)}\left(\begin{smallmatrix}x&y^i\\0&x\end{smallmatrix}\right)\oplus R/(xy,xh).
$$
As $(x)\ne(xy,xh)$, we have $R/(x)\ncong R/(xy,xh)$ and hence $R/(x)$ is a direct summand of $\cok_{R/(xy)}\left(\begin{smallmatrix}x&y^i\\0&x\end{smallmatrix}\right)$.
Note that $R/(xy)=S/(x^2h,xy)=S/(x^2(x^2s+xyt+y^2u),xy)=S/(x^4s,xy)$.
Put $T:=R/(xy,x^4)=S/(x^4,xy)$.
Applying the functor $-\otimes_RR/(x^4)$, we see that $T/(x)=R/(x)$ is a direct summand of $L:=\cok_{T}\left(\begin{smallmatrix}x&y^i\\0&x\end{smallmatrix}\right)$.
Write $L=T/(x)\oplus D$ with $D\in\mod T$.
It is easy to verify that the sequence
$$
0\gets L\gets T^{\oplus2}\xleftarrow{\left(\begin{smallmatrix}x&y^i\\0&x\end{smallmatrix}\right)}T^{\oplus2}\xleftarrow{\left(\begin{smallmatrix}y&x^3&0\\0&0&x^3\end{smallmatrix}\right)}T^{\oplus3}
$$
is exact, and we observe $D\cong T/(v)$ for some $v\in T$.
Uniqueness of a minimal free resolution gives rise to a commutative diagram
$$
\xymatrix{
0 & L\ar[l]\ar[d]_\cong & T^{\oplus2}\ar[l]\ar[d]_\cong^{\left(\begin{smallmatrix}a_1&a_2\\a_3&a_4\end{smallmatrix}\right)} && T^{\oplus2}\ar[ll]_{\left(\begin{smallmatrix}x&y^i\\0&x\end{smallmatrix}\right)}\ar[d]_\cong^{\left(\begin{smallmatrix}b_1&b_2\\b_3&b_4\end{smallmatrix}\right)} && T^{\oplus3}\ar[ll]_{\left(\begin{smallmatrix}y&x^3&0\\0&0&x^3\end{smallmatrix}\right)}\\
0 & T/(x)\oplus T/(v)\ar[l] & T^{\oplus2}\ar[l] && T^{\oplus2}\ar[ll]^{\left(\begin{smallmatrix}x&0\\0&v\end{smallmatrix}\right)}
}
$$
with vertical maps being isomorphisms.
The elements $a_1a_4-a_2a_3$ and $b_1b_4-b_2b_3$ are units of $T$.
We have $a_1y^i+a_2x=xb_2$ and $a_1x=xb_1$ in $T$.
Hence $a_1y^i\in(x)\in\spec T$, which implies $a_1\in(x)$.
Also, $a_1-b_1\in(0:x)=(x^3,y)$, which implies $b_1\in(x,y)$.
It follows that $a_2,a_3,b_2,b_3$ are units of $T$.
The equality $a_3x=vb_3$ implies that $(x)=(v)$ in $T$.
We obtain isomorphisms
$$
T/(x^3,y)\oplus T/(x^3)
\cong\cok_T\left(\begin{smallmatrix}y&x^3&0\\0&0&x^3\end{smallmatrix}\right)
\cong\syz_TL
\cong(x)\oplus(v)
\cong(x)^{\oplus2}
\cong(T/(x^3,y))^{\oplus2}.
$$
It follows that $(x^3)=(x^3,y)$ in $T$, which is a contradiction.
Consequently, $C$ is isomorphic to $R/(xh)$.

(2) Next we consider the case where $C$ is a direct summand of $N_i$.
The proof is analogous to that of (1).
Again, assume $C\ncong R/(xh)$.
Then $C\cong R/(x)$ by Lemma \ref{36}.
Set $T:=R/(xh)=S/(xh)$.
Applying $-\otimes_RT$, we see that $R/(x)=T/(x)$ is a direct summand of
$$
N_i/xhN_i
=\cok_{T}\left(\begin{smallmatrix}0&-y^ih&y^{i+1}\\0&0&x\\0&0&-xy\end{smallmatrix}\right)
\cong\cok_{T}\left(\begin{smallmatrix}0&-y^ih&y^{i+1}\\0&0&x\\0&0&0\end{smallmatrix}\right)
\cong T\oplus\cok_T\left(\begin{smallmatrix}y^ih&y^{i+1}\\0&x\end{smallmatrix}\right),
$$
which implies that $T/(x)$ is a direct summand of $L:=\cok_{T}\left(\begin{smallmatrix}y^ih&y^{i+1}\\0&x\end{smallmatrix}\right)$.
There are an isomorphism $L\cong T/(x)\oplus T/(v)$ with $v\in T$ and a commutative diagram:
$$
\xymatrix{
0 & L\ar[l]\ar[d]_\cong & T^{\oplus2}\ar[l]\ar[d]_\cong^{\left(\begin{smallmatrix}a_1&a_2\\a_3&a_4\end{smallmatrix}\right)} && T^{\oplus2}\ar[ll]_{\left(\begin{smallmatrix}y^ih&y^{i+1}\\0&x\end{smallmatrix}\right)}\ar[d]_\cong^{\left(\begin{smallmatrix}b_1&b_2\\b_3&b_4\end{smallmatrix}\right)} \\
0 & T/(x)\oplus T/(v)\ar[l] & T^{\oplus2}\ar[l] && T^{\oplus2}\ar[ll]^{\left(\begin{smallmatrix}x&0\\0&v\end{smallmatrix}\right)}
}
$$
Note that $\spec T=\{(x),(h),\n T\}$.
We have $(h)\ni a_1y^ih=xb_1\in(x)$, which implies $a_1\in(x)$ and $b_1\in(h)$.
As $a_1a_4-a_2a_3$ and $b_1b_4-b_2b_3$ are units, so are $a_2,a_3,b_2,b_3$.
The equalities $a_3y^ih=vb_3$ and $a_3y^{i+1}+a_4x=vb_4$ imply $a_3y^i(b_3^{-1}hb_4-y)=a_4x\in(x)$, which gives $b_3^{-1}hb_4-y\in(x)$.
Hence $y\in(x,h)=(x,x^2s+xyt+y^2u)=(x,y^2u)$ in $T$, which is a contradiction.
Thus $C\cong R/(xh)$.
\end{proof}

\begin{lem}[cf. Theorem \ref{33}]\label{35}
The ring $R$ has infinite $\cmp$-representation type.
\end{lem}

\begin{proof}
Assume contrarily that $R$ has finite $\cmp$-representation type.
Then, by (1) and (2) of Lemma \ref{37}, there exists an integer $a\ge1$ such that $M_i$ is decomposable for all integers $i\ge a$.

Suppose that for some $i\ge1$ the module $M_i$ has a cyclic direct summand $C\in\cmp(R)$.
Then $C$ is isomorphic to $R/(xh)$ by Lemma \ref{38}, and $\syz_RC=R/(x)$ is a direct summand of $\syz_RM_i=N_i$ by Lemma \ref{37}(1).
Applying Lemma \ref{38} again, we have to have $R/(x)\cong R/(xh)$, which is a contradiction.

Thus $M_i$ has no cyclic direct summand belonging to $\cmp(R)$ for all $i\ge 1$.
This means that for every $i\ge a$ the $R$-module $M_i$ has an indecomposable direct summand $Y_i\in\cmp(R)$ with $\nu(Y_i)=2$.
This, in turn, contradicts the assumption that $R$ has finite $\cmp$-representation type.
\end{proof}

Now the purpose of this subsection is readily accomplished:

\begin{proof}[Proof of Theorem \ref{34}]
The theorem is an immediate consequence of Lemma \ref{35} and what we state just after the theorem.
\end{proof}

\section{On the higher-dimensional case}\label{h}

In this section, we explore the higher-dimensional case: we consider Cohen--Macaulay local rings $R$ with $\dim R\ge2$ and having finite $\cmp$-representation type.
In particular, we give various results supporting Conjecture \ref{conj11}.
We begin with presenting an example by using a result obtained in Section 4.

\begin{ex}
Let $S$ be a regular local ring with a regular system of parameters $x,y,z$.
Then $R=S/(xyz)$ has infinite $\cmp$-representation type.
\end{ex}

\begin{proof}
Let $I=(xy)$ be an ideal of $R$.
Then $(0:I)=(z)$ in $R$, and $\height(I+(0:I))=\height(xy,z)=1<2=\dim R$.
The ring $R/I=S/(xy)$ is a $2$-dimensional hypersurface which does not have an isolated singularity.
We see by \cite[Corollary 2]{HL02} that $R/I$ has infinite $\cm$-representation type.
It follows from Theorem \ref{19}(3a) that $R$ has infinite $\cmp$-representation type.
\end{proof}

\begin{rem}

We remark that the indecomposables in $\cmz$(R) for $R=k[\![ x,y,z ]\!]/(xyz)$ have been classified by Burban and Drozd (see \cite[Theorem 8.6]{BD}).

\end{rem}

We consider constructing from a given hypersurface of infinite $\cmp$-representation type another hypersurface of infinite $\cmp$-representation type.
For this we establish the following lemma, which provides a version of Kn\"{o}rrer's periodicity theorem for $\cmp(R)$.

\begin{lem}\label{40}
Let $(S,\n)$ be a regular local ring, and let $f,g\in S$.
Let $R=S/(f)$ and $R^\sharp=S[\![x]\!]/(f+x^2g)$ be hypersurfaces with $x$ an indeterminate over $S$.
Then the following statements hold.
\begin{enumerate}[\rm(1)]
\item
There is an additive functor
$$
\Phi:\cmp(R)\to\cmp(R^\sharp),\qquad
\cok(A,B)\mapsto\cok\left(\left(\begin{smallmatrix}A&-xE\\xgE&B\end{smallmatrix}\right),\left(\begin{smallmatrix}B&xE\\-xgE&A\end{smallmatrix}\right)\right).
$$
\item
Let $M\in\ind\cmp(R)$ and put $N=\Phi(M)$.
Then one has either $N\in\ind\cmp(R^\sharp)$ or $N\cong X\oplus Y$ for some $X,Y\in\ind\cmp(R^\sharp)$.
\end{enumerate}
\end{lem}

\begin{proof}
(1) It holds that $\left(\begin{smallmatrix}A&-xE\\xgE&B\end{smallmatrix}\right)\left(\begin{smallmatrix}B&xE\\-xgE&A\end{smallmatrix}\right)=\left(\begin{smallmatrix}B&xE\\-xgE&A\end{smallmatrix}\right)\left(\begin{smallmatrix}A&-xE\\xgE&B\end{smallmatrix}\right)=(f+x^2g)E$.
If $(V,W):(A,B)\to(A',B')$ is a morphism of matrix factorizations of $f$ over $S$, then $\left(\left(\begin{smallmatrix}V&0\\0&W\end{smallmatrix}\right),\left(\begin{smallmatrix}W&0\\0&V\end{smallmatrix}\right)\right):\left(\left(\begin{smallmatrix}A&-xE\\xgE&B\end{smallmatrix}\right),\left(\begin{smallmatrix}B&xE\\-xgE&A\end{smallmatrix}\right)\right)\to\left(\left(\begin{smallmatrix}A'&-xE\\xgE&B'\end{smallmatrix}\right),\left(\begin{smallmatrix}B'&xE\\-xgE&A'\end{smallmatrix}\right)\right)$ is a morphism of matrix factorizations of $f+x^2g$ over $S[\![x]\!]$.
We observe that $\Phi$ defines an additive functor from $\cm(R)$ to $\cm(R^\sharp)$.

Fix $M\in\cmp(R)$.
Let $(A,B)$ be the corresponding matrix factorization.
Set $N=\cok_{S[\![x]\!]}\left(\begin{smallmatrix}A&-xE\\xgE&B\end{smallmatrix}\right)$.
There is a nonmaximal prime ideal $\p$ of $S$ such that $M_\p$ is not $R_\p$-free.
Put $\q=\p S[\![x]\!]+xS[\![x]\!]$.
We see that $\q$ is a nonmaximal prime ideal of $S[\![x]\!]$.
Suppose that $N_\q\cong(R^\sharp)_\q^{\oplus n}$ for some $n$.
Then
$$
R_\p^{\oplus n}
\cong((R^\sharp/xR^\sharp)^{\oplus n})_\q
\cong N_\q/xN_\q
\cong\cok_{S[\![x]\!]_\q}\left(\begin{smallmatrix}A&0\\0&B\end{smallmatrix}\right)
\cong\cok_{S_\p}A\oplus\cok_{S_\p}B
\cong M_\p\oplus(\syz_RM)_\p,
$$
which implies that $M_\p$ is $R_\p$-free, a contradiction.
Therefore $N_\q$ is not $(R^\sharp)_\q$-free, and we obtain $N\in\cmp(R^\sharp)$.
Thus $\Phi$ induces an additive functor from $\cmp(R)$ to $\cmp(R^\sharp)$.

(2) Let $(A,B)$ be the matrix factorization which gives $M$.
Then $N=\cok_{S[\![x]\!]}\left(\begin{smallmatrix}A&-xE\\xgE&B\end{smallmatrix}\right)$.
Suppose that $N$ is decomposable.
Then $N\cong X\oplus Y$ for some nonzero modules $X,Y\in\cm(R^\sharp)$.
It holds that
$$
X/xX\oplus Y/xY
\cong N/xN
\cong\cok_{S}\left(\begin{smallmatrix}A&0\\0&B\end{smallmatrix}\right)
\cong\cok_SA\oplus\cok_SB
\cong M\oplus\syz_RM.
$$
Since $R$ is Gorenstein, not only $M$ but also $\syz_RM$ is indecomposable; see \cite[Lemma 8.17]{Y}.
Nakayama's lemma guarantees that $X/xX$ and $Y/xY$ are nonzero, and both $X$ and $Y$ have to be indecomposable.
We may assume that $M\cong X/xX$ and $\syz_RM\cong Y/xY$.
Take a nonmaximal prime ideal $\p$ of $S$ such that $M_\p$ is not $R_\p$-free.
Then $\q:=\p S[\![x]\!]+xS[\![x]\!]$ is a nonmaximal prime ideal of $S[\![x]\!]$ as in the proof of (1).
We easily see that the $R_\p$-module $(\syz_RM)_\p$ is not free.
Now it follows that neither $X_\q$ nor $Y_\q$ is free over $(R^\sharp)_\q$, which shows that $X,Y\in\cmp(R^\sharp)$.
\end{proof}

Infinite $\cmp$-representation type ascends from $R$ to $R^\sharp$.

\begin{prop}\label{47}
Let $(S,\n)$ be a regular local ring and $f,g\in S$.
Let $R=S/(f)$ and $R^\sharp=S[\![x]\!]/(f+x^2g)$ be hypersurfaces with $x$ an indeterminate.
If $R$ has infinite $\cmp$-representation type, then so does $R^\sharp$.
\end{prop}

\begin{proof}
Pick any $M_1\in\ind\cmp(R)$.
The set $\ind\cmp(R)\setminus\{M_1,\syz M_1\}$ is infinite, and we pick any $M_2$ in this set.
The set $\ind\cmp(R)\setminus\{M_1,\syz M_1,M_2,\syz M_2\}$ is infinite, and we pick any $M_3$ in it.
Iterating this procedure, we obtain modules $M_1,M_2,M_3,\dots$ in $\ind\cmp(R)$ such that $M_i\ncong M_j$ and $M_i\cong\syz M_j$ for all $i\ne j$.
We put $N_i=\Phi M_i$ for each $i$, where $\Phi$ is the functor defined in Lemma \ref{40}.
Then by the lemma $N_i$ is either in $\ind\cmp(R^\sharp)$ or isomorphic to $X_i\oplus Y_i$ for some $X_i,Y_i\in\ind\cmp(R^\sharp)$.

Assume $N_i\cong N_j$ for some $i\ne j$.
Then, as we saw in the proof of the lemma, there are isomorphisms $M_i\oplus\syz M_i\cong N_i/xN_i\cong N_j/xN_j\cong M_j\oplus\syz M_j$ and the modules $M_i,\syz M_i,M_j,\syz M_j$ are indecomposable.
This contradicts the choice of these modules.
Hence we have $N_i\ncong N_j$ for all $i\ne j$.

Suppose that there are only a finite number, say $n$, of indecomposable modules in $\cmp(R)$.
Then it is seen that the set $\{N_1,N_2,N_3,\dots\}/_{\cong}$ has cardinality at most $n+\binom{n+1}{2}$, which is a contradiction.
We now conclude that $R^\sharp$ has infinite $\cmp$-representation type, and the proof of the proposition is completed.
\end{proof}

Here is an application of Proposition \ref{47}.

\begin{cor}
Let $R$ be a $2$-dimensional complete local hypersurface with algebraically closed residue field $k$ of characteristic $0$ and not having an isolated singularity.
Suppose that $R$ has multiplicity at most $2$.
If $R$ has finite $\cmp$-representation type, then $R\cong k[\![x,y,z]\!]/(f)$ with $f=x^2+y^2$ or $f=x^2+y^2z$, and hence $R$ has countable $\cm$-representation type.
\end{cor}

\begin{proof}
If $\e(R)=1$, then $R$ is regular, which contradicts the assumption that $R$ does not have an isolated singularity.
Hence $\e(R)=2$, and the combination of Cohen's structure theorem and the Weierstrass preparation theorem shows $R\cong k[\![x,y,z]\!]/(x^2+g)$ for some $g\in k[\![y,z]\!]$; see \cite[Proof of Theorem 8.8]{Y}.
It follows from Proposition \ref{47} that the $1$-dimensional hypersurface $S:=k[\![y,z]\!]/(g)$ has finite $\cmp$-representation type.
By virtue of Theorem \ref{46}, we obtain $g=y^2$ or $g=y^2z$ after changing variables (i.e., after applying a $k$-algebra automorphism of $k[\![y,z]\!]$).
We observe that $R$ is isomorphic to either $k[\![x,y,z]\!]/(x^2+y^2)$ or $k[\![x,y,z]\!]/(x^2+y^2z)$.
It follows from \cite[Propositions 14.17 and 14.19]{LW} that $R$ has countable $\cm$-representation type.
\end{proof}

Proposition \ref{47} can provide a lot of examples of hypersurfaces of infinite $\cmp$-representation type of higher dimension.
The following example is not covered by this proposition or any other general result given in this paper.

\begin{ex}\label{57}
Let $S$ be a regular local ring with a regular system of parameters $x,y,z$.
Let
$$
f=x^n+x^2ya+y^2b
$$
be an irreducible element of $S$ with $n\ge4$ and $a,b\in S$.
Then the hypersurface $R=S/(f)$ has infinite $\cmp$-representation type.
\end{ex}

\begin{proof}
Putting $g=x^2a+yb$, we have $f=x^n+yg$.
For each integer $i\ge0$ we define a pair of matrices $A_i=
\left(\begin{smallmatrix}
x^2&xz^i\\
0&-x^2
\end{smallmatrix}\right)$ and $B_i=
\left(\begin{smallmatrix}
x^{n-2}&x^{n-3}z^i\\
0&-x^{n-2}
\end{smallmatrix}\right)$, which gives a matrix factorization of $x^n$ over $S$ and $S/(y)$.
Define another pair of matrices $A_i^\sharp=
\left(\begin{smallmatrix}
A_i&-yE\\
gE&B_i
\end{smallmatrix}\right)$ and $B_i^\sharp=
\left(\begin{smallmatrix}
B_i&yE\\
-gE&A_i
\end{smallmatrix}\right)$.
These form a matrix factorization of $f$ over $S$, and hence $M_i:=\cok_S(A_i^\sharp)$ is a maximal Cohen--Macaulay $R$-module.
There are equalities
$$
\fitt_3^S(M_i)=\I_1(A_i^\sharp)=(x^2,xz^i,x^{n-2},x^{n-3}z^i,y,g)S=(x^2,xz^i,y)S
$$
of ideals of $S$, where we use $n\ge4$.

Suppose that $M_i\cong M_j$ for some $i<j$.
Then $(x^2,xz^i,y)S=(x^2,xz^j,y)S$ and $(x^2,xz^i)\overline{S}=(x^2,xz^j)\overline{S}$, where $\overline S:=S/(y)$ is a regular local ring having the regular system of parameters $x,z$.
Hence $z^i\in(x,z^j)\overline S$ and $z^i\in z^j\widetilde S$ where $\widetilde S:=\overline S/x\overline S$ is a discrete valuation ring with $z$ a uniformizer.
This gives a contradiction, and we see that $M_i\ncong M_j$ for all $i\ne j$.

Let $\p=(x,y)S\in\spec S$, and fix an integer $i\ge0$.
Note that all the entries of $A_i,B_i$ are in $\p$ since $n\ge4$.
It follows from \cite[Remark 7.5]{Y} that the $R_\p$-module $(M_i)_\p$ does not have a nonzero free summand.
Since $f$ is assumed to be irreducible, $R$ is an integral domain.
Hence each nonzero direct summand $X$ of the maximal Cohen--Macaulay $R$-module $M_i$ has positive rank, and hence has full support.
Therefore $X_\p\ne0$, and thus all the indecomposable direct summands of $M_i$ belong to $\ind\cmp(R)$.
Since all the $M_i$ are generated by four elements, it is observed that $\ind\cmp(R)$ is an infinite set.
\end{proof}

To prove our next result, we prepare a lemma on unique factorization domains.

\begin{lem}\label{21}
Let $R$ be a Cohen--Macaulay factorial local ring with $\dim R\ge3$.
Let $I$ be an ideal of $R$ generated by two elements.
Then $\depth R/I>0$.
\end{lem}

\begin{proof}
We write $I=(x,y)R$ and put $g=\gcd\{x,y\}$.
Then $x=gx'$ and $y=gy'$ for some $x',y'\in R$, and we set $I'=(x',y')R$.
There is an exact sequence $0\to R/I'\xrightarrow{g} R/I\to R/gR\to0$ of $R$-modules.
As $R$ is Cohen--Macaulay, we have $\depth R\ge3$ and $\height I'=\grade I'$.
Since $R$ is a domain and $g\ne0$, we have $\depth R/gR=\depth R-1\ge2$.
If $\height I'=1$, then $I'$ is contained in a principal prime ideal, which contradicts the fact that $x',y'$ are coprime.
Hence $\height I'=2$, and the sequence $x',y'$ is $R$-regular.
It follows that $\depth R/I'=\depth R-2\ge1$, and the depth lemma implies $\depth R/I\ge1$.
\end{proof}

Now we can prove the following theorem, which provides the shape of a hypersurface of infinite $\cmp$-representation type.

\begin{thm}\label{22}
Let $(S,\n)$ be a regular local ring and $x,y\in\n$.
Suppose that the ideal $(x,y)$ of $S$ is neither prime nor $\n$-primary.
Then $R=S/(xy)$ has infinite $\cmp$-representation type.
\end{thm}

\begin{proof}
Lemma \ref{21} guarantees that there exists an $S/(x,y)$-regular element $a\in\n$.
Take a minimal prime $\p$ of $(x,y)$.
Since $(x,y)$ is not prime, we can choose an element $b \in \p\setminus(x,y)$.
Set $z_n=a^nb$ for each $n$.
The matrices $A_n=
\left(\begin{smallmatrix}
x & z_n\\
0 & -y
\end{smallmatrix}\right)$ and $B_n=\left(\begin{smallmatrix}
y & z_n\\
0 & -x
\end{smallmatrix}\right)$ with $n\ge1$ form a matrix factorization of $xy$ over $S$, and $M_n=\cok_SA_n$ is a maximal Cohen--Macaulay $R$-module.
Put $I_n:=\I_1(A_n)=(x,y,z^n)\subseteq S$.
Since the $I_n$ are pairwise distinct, the $M_n$ are pairwise nonisomorphic.
If $M_n$ is decomposable, it decomposes into two cyclic $R$-modules, while Lemma \ref{20} says that there are only finitely many such cyclic modules up to isomorphism.
Thus we find infinitely many $n$ such that $M_n$ is indecomposable.
Since $(x,y,z_n)$ is contained in $\p$, each $(M_n)_{\p}$ has no nonzero free summand by \cite[(7.5.1)]{Y}.
In particular, we have $M_n\in\cmp(R)$.
Now it is seen that $R$ has infinite $\cmp$-representation type.
\end{proof}

Applying the above theorem, we can obtain a couple of restrictions for a hypersurface of dimension at least $2$ which is not an integral domain but has finite $\cmp$-representation type.

\begin{cor}\label{69}
Let $R$ be a complete local hypersurface of dimension $d\ge2$ which is not a domain.
Suppose that $R$ has finite $\cmp$-representation type.
Then one has $d=2$, and there exist a complete regular local ring $S$ of dimension $3$ and elements $x,y\in S$ satisfying the following conditions.
\begin{enumerate}[\rm(1)]
\item
$R$ is isomorphic to $S/(xy)$.
\item
$S/(x)$ and $S/(y)$ have finite $\cm$-representation type.
\item
$S/(x,y)$ is a domain of dimension $1$.
\end{enumerate}
\end{cor}

\begin{proof}
Corollary \ref{16}(1) says that $R$ satisfies Serre's condition $(\RR_{d-2})$.
Suppose $d\ge3$.
Then $R$ satisfies $(\RR_1)$, and hence it is normal.
In particular, $R$ is a domain, contrary to our assumption.
Therefore, we have to have $d=2$.
Cohen's structure theorem yields $R\cong S/fS$ for some $3$-dimensional complete regular local ring $(S,\n)$ and $f\in\n\setminus\n^2$.
As $R$ is not a domain, there are elements $x,y\in S$ with $f=xy$.
Since $\dim S=3$, the ideal $(x,y)S$ is not $\n$-primary.
Hence $\dim S/(x,y)S=1$, and $S/(x,y)S$ is a domain by Theorem \ref{22}.
We have $\dim R=\dim R/xR=2$, $(0:_Rx)=yR$ and $\height(xR+(0:_Rx))<2$.
It follows from Theorem \ref{19}(3a) that $S/xS$ has finite $\cm$-representation type, and similarly so does $S/yS$.
\end{proof}

Proposition \ref{47} gives an ascent property of infinite $\cmp$-representation type.
Now we presents a descent property of infinite $\cmp$-representation type.

\begin{thm}\label{52}
Let $\phi:(R,\m,k)\to(S,\n,l)$ be a finite local homomorphism of Cohen--Macaulay local rings of dimension $d$ such that $S$ is a domain.
Set $\p=\ker\phi$ and assume the following.
\begin{enumerate}[\quad\rm(a)]
\item
The induced embedding $R/\p\hookrightarrow S$ is birational.
\item
There exists $\q\in\v(\p)\setminus\{\m\}$ such that $R_\q$ is not a direct summand of $S_\q$.
\end{enumerate}
If $S$ has infinite $\cm$-representation type, then $R$ has infinite $\cmp$-representation type.
\end{thm}

\begin{proof}
We prove the theorem by establishing several claims.

\setcounter{claim}{0}
\begin{claim}\label{48}
Let $X\ne0$ be an $R$-submodule of a maximal Cohen--Macaulay $S$-module $M$.
Then $X_\q\ne0$.
\end{claim}

\begin{proof}[Proof of Claim]
Assume $X_\q=0$.
Then there exists an element $s\in\ann_RX$ such that $s\notin\q$.
As $\p\subseteq\q$, we have $s\notin\p$, which means $\phi(s)\ne0$.
Choose a nonzero element $x\in X$.
Since $s$ annihilates $X$, we have $0=s\cdot x=\phi(s)x$ in $M$.
This contradicts the fact that $M$ is torsion-free over the domain $S$.
\renewcommand{\qedsymbol}{$\square$}
\end{proof}

\begin{claim}\label{49}
Let $M\in\cmz(S)$.
Let $X$ be an indecomposable $R$-module which is a direct summand of $M$.
Then $X\in\ind\cmp(R)$.
\end{claim}

\begin{proof}[Proof of Claim]
As $\depth_RM=\depth_SM\ge d$, we have $M\in\cm(R)$ and hence $X\in\ind\cm(R)$.
To show the claim, it suffices to verify that $X_\q$ is not $R_\q$-free.

Take an exact sequence $\sigma:0\to\syz_SM\to S^{\oplus n}\to M\to0$.
Since $M$ belongs to $\cmz(S)$, the $S$-module $E:=\Ext_S^1(M,\syz_SM)$ has finite length.
The induced field extension $k\hookrightarrow l$ is finite because so is the homomorphism $\phi$, and hence $E$ also has finite length as an $R$-module.
As $\q$ is a nonmaximal prime ideal of $R$, we have $0=E_\q=\Ext_{S_\q}^1(M_\q,(\syz_SM)_\q)$, and the exact sequence $\sigma_\q:0\to(\syz_SM)_\q\to S_\q^{\oplus n}\to M_\q\to0$ corresponds to an element in this Ext module.
Hence $\sigma_\q$ has to split, and $M_\q$ is a direct summand of $S_\q^{\oplus n}$ as an $S_\q$-module.
(Note that $S_\q$ is not necessarily a local ring.)
The $R_\q$-module $X_\q$ is a direct summand of $M_\q$, and is nonzero by Claim \ref{48}.

Suppose that $X_\q$ is $R_\q$-free.
Then $R_\q$ is a direct summand of $S_\q^{\oplus n}$ in $\mod R_\q$.
As $R_\q$ is a local ring, $R_\q$ is a direct summand of $S_\q$.
This contradicts the assumption of the theorem, and thus $X_\q$ is not $R_\q$-free.
\renewcommand{\qedsymbol}{$\square$}
\end{proof}

\begin{claim}\label{50}
One has the inclusion $\ind\cmz(S)\subseteq\ind\cmp(R)$.
\end{claim}

\begin{proof}[Proof of Claim]
Take $M\in\ind\cmz(S)$.
Lemma \ref{31} implies that $M$ is indecomposable as an $R/\p$-module, and it is indecomposable as an $R$-module.
Taking $X:=M$ in Claim \ref{49}, we have $M\in\ind\cmp(R)$.
\renewcommand{\qedsymbol}{$\square$}
\end{proof}

It follows from Lemma \ref{c2} that $S$ has infinite $\cmz$-representation type.
Claim \ref{50} implies that $R$ has infinite $\cmp$-representation type, and the proof of the theorem is completed.
\end{proof}

We obtain an application of the above theorem, which gives an answer to Question \ref{2}.
For a ring $R$ we denote by $\overline R$ the integral closure of $R$.
Recall that a typical example of a henselian Nagata ring is a complete local ring.

\begin{cor}\label{55}
Let $R$ be a $2$-dimensional henselian Nagata Cohen--Macaulay non-normal local ring.
Suppose that $R$ has finite $\cmp$-representation type.
Then the following statements hold.
\begin{enumerate}[\rm(1)]
\item
There exists a minimal prime $\p$ of $R$ such that the integral closure $\overline{R/\p}$ has finite $\cm$-representation type.
In particular, if $R$ is a domain, then $\overline R$ has finite $\cm$-representation type.
\item
If $R$ is Gorenstein, then $R$ is a hypersurface.
\end{enumerate}
\end{cor}

\begin{proof}
By Corollary \ref{16}(1) the singular locus of $R$ has dimension at most one, so that $R$ satisfies Serre's condition $(\RR_0)$.
As $R$ is Cohen--Macaulay, it is reduced.
Let $S=\overline{R}$ be the integral closure of $R$.
We have a decomposition $S=\overline{R/\p_1}\oplus\cdots\oplus\overline{R/\p_n}$ as $R$-modules, where $\Min R=\{\p_1,\dots,\p_n\}$ (see \cite[Corollary 2.1.13]{HS}).
Since $R$ is Nagata, the extension $R\subseteq S$ is finite.
The ring $S$ is normal and has dimension two, so it is Cohen--Macaulay.

We claim that if $\p$ is a nonmaximal prime ideal of $R$ such that $S_\p$ is $R_\p$-free, then $R_\p$ is a regular local ring.
In fact, if $\height\p=0$, then $R_\p$ is a field.
Let $\height\p=1$.
The induced map $\spec S\to\spec R$ is surjective, and we find a prime ideal $P$ of $S$ such that $P\cap R=\p$.
We easily see $\height P=1$.
As $S$ is normal, $S_P$ is regular.
The induced map $R_\p\to S_P$ factors as $R_\p\xrightarrow{a}S_\p\xrightarrow{b}S_P$, where $a$ is a finite free extension, and $b$ is flat since $S_P=(S_\p)_{PS_\p}$.
Hence $R_\p\to S_P$ is a flat local homomorphism.
As $S_P$ is regular, so is $R_\p$.

Since $R$ does not have an isolated singularity, there exists a nonmaximal prime ideal $\p$ of $R$ such that $R_\p$ is not regular.
The claim implies that $S_\p$ is not $R_\p$-free, whence $S\in\cmp(R)$.
There exists an integer $1\le l\le n$ such that $T:=\overline{R/\p_l}$ belongs to $\cmp(R)$.

Put $\p:=\p_l\in\Min R$.
The ring $R/\p$ is also Nagata, and the extension $R/\p\subseteq T$ is finite and birational.
The ring $T$ is a $2$-dimensional henselian normal local domain, whence it is a Cohen--Macaulay.
Choose a nonmaximal prime ideal $\q$ of $R$ such that $T_\q$ is not $R_\q$-free.
If $\p$ is not contained in $\q$, then $(R/\p)_\q=\kappa(\p)_\q=0$ and $T_\q=0$, which particularly says that $T_\q$ is $R_\q$-free, a contradiction.
Hence $\p\subseteq\q$.

Suppose that $R_\q$ is a direct summand of $T_\q$.
Then there is an isomorphism $T_\q\cong R_\q\oplus X$ of $R_\q$-modules.
Since $T_\q$ is annihilated by $\p$, so is $R_\q$.
We have ring extensions $R_\q=(R/\p)_\q\subseteq T_\q\subseteq\kappa(\p)$, which especially says that $R_\q$ is a domain and that $T_\q$ has rank one as an $R_\q$-module.
Hence the $R_\q$-module $X$ has rank zero, and it is easy to see that $X=0$.
We get $T_\q\cong R_\q$, which contradicts the choice of $\q$.
Consequently, $T_\q$ does not have a direct summand isomorphic to $R_\q$.

Now, application of Theorem \ref{52} proves the assertion (1).
To show (2), we consider the $T$-module $U=\syz_T^2(T/\m T)$.
Fix any nonzero direct summand $X$ of $U$ or $T$ in $\mod R$.
Note that $T=\overline{R/\p}$ is a torsion-free module over $R/\p$.
Since $U$ is a submodule of a nonzero free $T$-module, $U$ is also torsion-free over $R/\p$, and so is $X$.
We easily see from this that $X_\q\ne0$.
The module $X_\q$ is a direct summand of $U_\q\cong T_\q^{\oplus\edim R-1}$.
As $R_\q$ is not a direct summand of $T_\q$, it is not a direct summand of $X_\q$.
In particular, $X$ belongs to $\cmp(R)$.
Thus, all the indecomposable direct summands of $U$ and of $T$ in $\mod R$ belong to $\ind\cmp(R)$, and it follows from Lemma \ref{41} that they have complexity at most one.
Hence $U$ and $T$ have complexity at most one over $R$, and so does $T/\m T$.
We obtain $\cx_Rk\le1$, and $R$ is a hypersurface by \cite[Theorem 8.1.2]{A}.
\end{proof}

The above result yields a strong restriction for finite $\cmp$-representation type in dimension two.

\begin{cor}\label{56}
Let $R$ be a $2$-dimensional non-normal Gorenstein complete local ring.
If $R$ has finite $\cmp$-representation type, then the integral closure $\overline R$ has finite $\cm$-representation type.
\end{cor}

\begin{proof}
If $R$ is a domain, then the assertion follows from Corollary \ref{55}(1).
Hence let us assume that $R$ is not a domain.
By Corollary \ref{55}(2) the ring $R$ is a hypersurface.
We can apply Corollary \ref{69} to see that there exists a $3$-dimensional regular local ring $S$ and elements $x,y\in S$ such that $R$ is isomorphic to $S/(xy)$ and $S/(x),S/(y)$ have finite $\cm$-representation type.
Note by \cite[Corollary 2]{HL02} that $S/(x),S/(y)$ are normal.
As in the beginning of the proof of Corollary \ref{55}, the ring $R$ is reduced.
Hence $(x)\ne(y)$, and we have an isomorphism $\overline R\cong\overline{S/(x)}\times\overline{S/(y)}=S/(x)\times S/(y)$; see \cite[Corollary 2.1.13]{HS}.
There is a natural category equivalence $\mod\overline R\cong\mod S/(x)\times\mod S/(y)$, which induces a category equivalence $\cm(\overline R)\cong\cm(S/(x))\times\cm(S/(y))$.
It is observed from this that $\overline R$ has finite $\cm$-representation type.
\end{proof}

The converse of Corollary \ref{56} does not necessarily hold, as the following example says.

\begin{ex}
Let $R=k[\![x,y,z]\!]/(x^4-y^3z)$ be a quotient of the formal power series ring $k[\![x,y,z]\!]$ over a field $k$.
Then $R$ is a $2$-dimensional complete non-normal local hypersurface.
The assignment $x\mapsto s^3t,\,y\mapsto s^4,\,z\mapsto t^4$ gives an isomorphism from $R$ to the subring $S=k[\![s^4,s^3t,t^4]\!]$ of the formal power series ring $T=k[\![s,t]\!]$.
The integral closure of $S$ is the fourth Veronese subring $k[\![s^4,s^3t,s^2t^2,st^3,t^4]\!]$ of $T$, which has finite $\cm$-representation type by \cite[Theorem 6.3]{LW}.
Hence $\overline R$ has finite $\cm$-representation type.
However, as $x^4-y^3z=x^4+x^2y\cdot0+y^2(-yz)$, the ring $R$ does not have finite $\cmp$-representation type by Example \ref{57}.
\end{ex}

\begin{rem}
The integral closure has to actually be regular (under the assumptions of Corollary \ref{56}) provided that our conjecture that countable $\cm$-representation type is equivalent to finite $\cmp$-representation type holds true in this setting.
\end{rem}

\section*{Acknowledgments}
The authors are deeply indebted to Hailong Dao for asking them whether there exists a Cohen--Macaulay local ring of finite $\cmp$-representation type other than the hypersurfaces of type $(\A_\infty)$ and $(\D_\infty)$.
In fact, this question gave the authors a strong motivation for this work. We are also grateful to two anonymous referees whose suggestions greatly improved the paper, and we also thank Tokuji Araya for stimulating discussions.
Most of this work was done during the visit of Toshinori Kobayashi to the University of Kansas in 2018--2019.
He is grateful to the Department of Mathematics for their hospitality.


\end{document}